\RequirePackage{fix-cm}
\documentclass[smallextended]{svjour3}       
\smartqed  

 \usepackage{mathptmx}      
\usepackage[T1]{fontenc}
\usepackage[utf8]{inputenc}

\usepackage{amsfonts,amssymb,amsmath,hyperref,graphicx,enumerate,wasysym,xcolor,fancyvrb}
\usepackage{autonum}
\usepackage{mathtools}
\usepackage{mathrsfs}
\numberwithin{equation}{section}

\newcommand{\bu}{{\bf u}}

\newcommand{\R}{{\mathbb R}}

\newcommand{\N}{{\mathbb N}}
\newcommand{\Z}{\mathbb{Z}}

\newcommand{\CC}{\mathcal{C}}

\newcommand{\CV}{\mathcal{V}}

\newcommand{\CF}{\mathcal{F}}

\newcommand{\scrH}{\mathscr{H}}
\newcommand{\scrL}{\mathscr{L}}

\newcommand{\Fu}{\mathfrak{u}}

\newcommand{\Fc}{\mathfrak{c}}

\newcommand{\Fm}{\mathfrak{m}}

\newcommand{\Fq}{\mathfrak{q}}
\newcommand{\Fv}{\mathfrak{v}}

\newcommand{\Fs}{\mathfrak{s}}

\newcommand{\eps}{\varepsilon}
\newcommand{\loc}{\mathrm{loc}}
\newcommand{\supp}{\mathrm{supp}}
\newcommand{\sym}{\mathrm{sym}}

\newtheorem{theorembis}{Theorem}

\newtheorem{asu}{}

\numberwithin{proposition}{section}
\numberwithin{lemma}{section}
\numberwithin{remark}{section}
\numberwithin{definition}{section}
\numberwithin{example}{section}
\numberwithin{corollary}{section}
\numberwithin{figure}{section}

\begin{document}

\title{Non-minimizing connecting orbits for multi-well systems
}


\author{Ramon Oliver-Bonafoux}


\institute{Ramon Oliver-Bonafoux \at
              Sorbonne Université, Laboratoire Jacques-Louis Lions. 4 Place Jussieu, 75005 Paris (France). \\
              \email{ramon.oliver\_bonafoux@sorbonne-universite.fr}           
}

\date{Received: date / Accepted: date}

\maketitle

\begin{abstract}
Given a nonnegative, smooth potential $V: \R^k \to \R$ ($k \geq 2$) with multiple zeros, we say that a curve $\Fq: \R \to \R^k$ is a \textit{connecting orbit} if it solves the autonomous system of ordinary differential equations
\begin{equation}
\Fq''= \nabla_{\bu} V(\Fq) , \hspace{2mm} \mbox{ in } \R
\end{equation}
and tends to a zero of $V$ at $\pm \infty$. Broadly, our goal is to study the existence of connecting orbits for the problem above using variational methods. Despite the rich previous literature concerning the existence of connecting orbits for other types of second order systems, to our knowledge only connecting orbits which minimize the associated energy functional in a suitable function space were proven to exist for autonomous multi-well potentials. The contribution of this paper is to provide, for a class of such potentials, some existence results regarding \textit{non-minimizing} connecting orbits. Our results are closely related to the ones in the same spirit obtained by J. Bisgard in his PhD thesis (University of Wisconsin-Madison, 2005), where non-autonomous periodic multi-well potentials (ultimately excluding autonomous potentials) are considered. Our approach is based on several refined versions of the classical Mountain Pass Lemma.
\end{abstract}

\section{Introduction}
The focus of this paper is to find solutions $\Fq: \R \to \R^k$ to the second order ordinary differential equation
\begin{equation}\label{connections-eq}
\Fq''= \nabla_{\bu} V(\Fq) , \hspace{2mm} \mbox{ in } \R
\end{equation}
verifying the conditions at infinity
\begin{equation}\label{infinity-conditions}
\lim_{t \to \pm \infty}\Fq(t) = \sigma_\pm.
\end{equation}
 If $\sigma_-=\sigma_+$, we say that the solution $\Fq$ is a \textit{homoclinic orbit}. If $\sigma_-\not = \sigma_+$, we say that $\Fq$ is a \textit{heteroclinic orbit}. The function $V$ is a standard multi-well potential. That is, a non negative function vanishing in a finite set $\Sigma$, with non degenerate global minima. The elements $\sigma_-$ and $\sigma_+$ belong to the set $\Sigma$. If $\sigma \in \Sigma$, we say that $\sigma$ is a \textit{well} of $V$. More precisely, $V$ is as follows:
\begin{asu}\label{asu-sigma}
$V \in \CC^2_{\loc}(\R^k)$ and  $V \geq 0$ in $\R^k$. Moreover, $V(\bu) = 0$ if and only if $\bu \in \Sigma$, where, for some $l \geq 2$
\begin{equation}\label{DEF-Sigma}
\Sigma:= \{\sigma_1,\ldots ,\sigma_l\}.
\end{equation}
\end{asu}
\begin{asu}\label{asu-infinity}
There exist $\alpha_0, \beta_0, R_0 > 0$ such that for all $\bu \in \R^k$ with $\lvert \bu \rvert \geq R_0$ it holds $ \langle \nabla_\bu V(\bu), \bu \rangle \geq \alpha_0 \lvert \bu \rvert^2$ and $V(\bu) \geq \beta_0$.
\end{asu}
\begin{asu}\label{asu-zeros}
For all $\sigma \in \Sigma$, the matrix $D^2V(\sigma)$ is positive definite.
\end{asu}
One formally checks that critical points of the functional
\begin{equation}
E(q) := \int_\R e(q)(t) dt := \int_\R \left[ \frac{1}{2}\lvert q'(t) \rvert^2 + V(q(t)) \right]dt, \hspace{2mm} q \in H^1_{\loc}(\R,\R^k),
\end{equation}
solve equation (\ref{connections-eq}). For any $(\sigma_i,\sigma_j) \in \Sigma^2$ we consider as in Rabinowitz \cite{rabinowitz93} the function space
\begin{equation}\label{set-heteroclinic}
X (\sigma_i,\sigma_j):= \left\{ q \in H^1_{\mathrm{loc}}(\R,\R^k): E(q) < +\infty  \mbox{ and } \lim_{t \to - \infty} q(t) = \sigma_i, \lim_{t \to +\infty}q(t)=\sigma_j\right\},
\end{equation}
and seek for critical points inside these spaces, as one easily shows that any finite energy curve in $H_\loc^1(\R,\R^k)$ must belong to $X(\sigma_i,\sigma_j)$ for some $(\sigma_i,\sigma_j) \in \Sigma^2$. We first define the infimum value
\begin{equation}\label{mij}
\Fm_{\sigma_i\sigma_j}:= \inf \{ E(q): q \in X(\sigma_i,\sigma_j)\}.
\end{equation}
The minimization problem in (\ref{mij}) is well understood. Indeed, if $\sigma_i=\sigma_j$, then (\ref{mij}) is attained by the constant curve $\sigma_i$. Otherwise, the problem is more involved but still well known (see Bolotin  \cite{bolotin}, Bolotin and Kozlov \cite{bolotin-kozlov}, Bertotti and Montecchiari \cite{bertotti-montecchiari} and Rabinowitz \cite{rabinowitz89,rabinowitz92}). Its lack of compactness implies that (\ref{mij}) does not always have a solution if $\Sigma$ possesses at least three elements. Let us fix once and for all $(\sigma^-,\sigma^+) \in \Sigma^2$, $\sigma^- \not = \sigma^+$ and set
\begin{equation}\label{Fm}
\Fm:=\mathfrak{m}_{\sigma^-\sigma^+}.
\end{equation}
We will assume that the following strict triangle's inequality holds:
\begin{asu}\label{asu-wells}
We have that
\begin{equation}
\forall \sigma \in \Sigma \setminus \{ \sigma^-,\sigma^+ \}, \hspace{2mm} \Fm < \Fm_{\sigma^-\sigma}+\Fm_{\sigma\sigma^+}.
\end{equation}
\end{asu}
Under assumption \ref{asu-wells}, it is well known that by concentration-compactness arguments (Lions \cite{lions}) there exists a globally minimizing heteroclinic in $X(\sigma^-,\sigma^+)$. See Theorem \ref{THEOREM-0} later for a precise statement.

We finally recall that the Sobolev embeddings imply that curves in $H^1_\loc(\R,\R^k)$ are continuous. This classical fact is used implicitly in the paper.
\subsection{Goal of the paper and statement of the main results}
The goal of this paper is to show that for a class of multi-well potentials $V$, there exist connecting orbits (either heteroclinic or homoclinic) which are not global minimizers in their natural spaces. We obtain several such results using variational methods. In particular our proof is based on a mountain pass argument (see Ambrosetti and Rabinowitz \cite{ambrosetti-rabinowitz}).

There exists a  vast literature concerning the existence of non-minimizing heteroclinics or homoclinic orbits for second order ordinary differential systems using variational methods. Some early references are Ambrosetti and Coti Zelati \cite{ambrosetti-coti zelati}, Coti Zelati and Rabinowitz \cite{coti zelati-rabinowitz}, Rabinowitz \cite{rabinowitz90,rabinowitz93}. Despite this fact, this question had not been addressed for the case of the autonomous multi-well potentials that we consider in this paper. However, the case of time-periodic multi-well potentials has been studied by Montecchiari and Rabinowitz in \cite{montecchiari-rabinowitz18,montecchiari-rabinowitz20} as well as by Bisgard in the second chapter of his PhD Thesis \cite{bisgard}. The present paper deals with a problem which is analogous to that in \cite{bisgard}. It is worth mentioning that while most of Bisgard's technical results also apply to the autonomous problem, his main results ultimately exclude such a possibility. The reason is that his key assumption is never satisfied by autonomous potentials due to the translation invariance of the associated problem. Roughly speaking, our Theorem \ref{THEOREM-general} shows that the ideas and arguments of Bisgard, as well as his key assumption, can be adapted to the autonomous setting. Nevertheless, our strategy and assumptions present some difference with respect to his. A detailed account regarding the main differences and similarities between the proofs is given in subsection \ref{SUBS-previous}.  We also provide the proof of other results, which are Theorems \ref{THEOREM-symmetry} and \ref{THEOREM-symmetry-final}, using for them a symmetry assumption on $V$. These results do not have a counterpart in Bisgard's work. 

Our mountain pass argument is carried out under a multiplicity assumption (up to translations) on the set of globally minimizing heteroclinics joining the two fixed wells $\sigma^-$ and $\sigma^+$. More precisely, the natural idea is to suppose that there exists a gap in the set of global minimizers and consider the family of paths that join two disconnected components. Subsequently, one shows that the associated min-max value is strictly larger than the minimum value, so that the existence of a mountain pass geometry has been established. Examples of earlier papers in which this approach is used are Bolotin and Rabinowitz \cite{bolotin-rabinowitz06,bolotin-rabinowitz07}, de la Llave and Valdinoci \cite{llave-valdinoci} as well as the above mentioned \cite{bisgard,montecchiari-rabinowitz18,montecchiari-rabinowitz20}. In our precise context, we work under assumption \ref{asu-heteroclinics}. This assumption was introduced by Alessio \cite{alessio} and it has been used under different forms for proving existence of solutions for Allen-Cahn systems, see the recent paper by Alessio and Montecchiari \cite{alessio-montecchiari} for a survey. It is the natural generalization of the assumption introduced by Alama, Bronsard and Gui \cite{alama-bronsard-gui} in their celebrated paper concerning entire solutions for two-dimensional Allen-Cahn systems. 

We write $\scrH:= H^1(\R,\R^k)$ and $\scrL:=L^2(\R,\R^k)$. We define
\begin{equation}
\CF:=\{ \Fq: \Fq \in X(\sigma^-,\sigma^+) \mbox{ and } E(\Fq)=\Fm\},
\end{equation}
the set of globally minimizing heteroclinics. The quantity $\Fm$ is as in (\ref{Fm}). The invariance by translations of the problem implies that if $\Fq \in \CF$, then for all $\tau \in \R$ we have $\Fq(\cdot+\tau) \in \CF$. It is well-known (see Lemma \ref{LEMMA-difference}) that $X(\sigma^-,\sigma^+)$ has the structure of an affine space in $H^1_{\loc}(\R,\R^k)$ and it is a metric space when endowed with the natural distance
\begin{equation}\label{distance}
d: (q,\tilde{q})^2 \in X(\sigma^-,\sigma^+) \to \lVert q-\tilde{q} \rVert_{\scrH}.
\end{equation}
We can now state the following assumption:
\begin{asu}\label{asu-heteroclinics}
It holds $\CF:= \CF_0 \cup \CF_1$ where $\CF_0$ and $\CF_1$ are not empty and such that 
\begin{equation}
d(\CF_0,\CF_1) > 0,
\end{equation}
where $d$ is the distance defined in (\ref{distance}).
\end{asu}
As stated before Assumption \ref{asu-heteroclinics} is the gap condition which permits the mountain pass approach. Implicitly, it implies that $k \geq 2$,  as it is well-known that heteroclinics are unique in the scalar case $k=1$. As it was pointed out before, \ref{asu-heteroclinics} was already considered in \cite{alessio} and it generalizes the one made in the previous work \cite{alama-bronsard-gui}. Let us now define
\begin{equation}\label{psi}
\psi(t):= \begin{cases}
\sigma^- &\mbox{ if } t \leq -1,\\
\frac{t+1}{2}\sigma^++\frac{1-t}{2}\sigma^- &\mbox{ if } -1 \leq t\leq 1,\\
\sigma^+ &\mbox{ if } t \geq 1.
\end{cases}
\end{equation}
We have that for all $v \in \scrH$ it holds that $v+\psi \in X(\sigma^-,\sigma^+)$ (see Lemma \ref{LEMMA-difference} for a proof). As in the earlier works \cite{bisgard,montecchiari-rabinowitz18} we define the functional
\begin{equation}\label{FunctionalJ}
J: v \in \scrH \to E(v+\psi) \in \R,
\end{equation}
which presents the advantage of being defined in a linear space. We also point out that the choice of the function $\psi$ is arbitrary.
\subsubsection{The general case}
We set $\CV:=\CF-\{\psi\}$, and for $i \in \{0,1\}$, $\CV_i:= \CF_i-\{\psi\}$. Those are nonempty subsets of $\scrH$. We can now define the mountain pass family:
\begin{equation}\label{Gamma}
\Gamma:= \{ \gamma \in C([0,1],\scrH): \forall i \in \{0,1\},  \gamma(i) \in \CV_i\}
\end{equation}
and the corresponding mountain pass value
\begin{equation}\label{Fc}
\Fc:= \inf_{\gamma \in \Gamma} \max_{s \in [0,1]}J(\gamma(s)) < +\infty.
\end{equation}
In this paper we show that $\Fc>\Fm$ (see Proposition \ref{PROPOSITION-mountain-pass} later). Therefore, $\Fc$ is a \textit{mountain pass value} for $J$. As it is well known, this is generally not sufficient to ensure the existence of new solutions. In order to prove our first result, we will need two more assumptions:
\begin{asu}\label{asu-otherwells}
It holds that $\Fc <\Fm^\star$,where
\begin{equation}\label{Fm_star}
\Fm^\star:=\min\{\Fm_{\sigma^-\sigma}+\Fm_{\sigma\sigma^+}:\sigma \in \Sigma \setminus \{ \sigma^-,\sigma^+\}\}.
\end{equation}
\end{asu}
It is clear that \ref{asu-otherwells} is stronger than \ref{asu-wells} and weaker than $\Sigma=\{\sigma^-,\sigma^+\}$. It is used in order to prevent that curves with energy close to $\Fc$ go trough a well in $\Sigma \setminus \{\sigma^-,\sigma^+\}$, in case there are any.
\begin{asu}\label{asu-K}
There exists a closed set $K \subset \R^k$ such that:
\begin{enumerate}
\item There exists $\nu_0>0$ such that 
\begin{equation}
\forall \Fq \in \CF, \hspace{2mm} \mathrm{dist}(\Fq(\R),K) \geq \nu_0
\end{equation}
where $\mathrm{dist}$ stands for the usual Euclidean distance between two sets in $\R^k$.
\item There exists $M > \Fc$ such that for any $\gamma \in \Gamma$ ($\Gamma$ is defined in (\ref{Gamma})) such that $\max_{s \in [0,1]}J(\gamma(s)) \leq M$, there exists $s_\gamma \in [0,1]$ such that $J(\gamma(s_\gamma)) \geq \Fc$ and $(\gamma(s_\gamma)+\psi)(\R) \cap K \not = \emptyset$, where $\Fc$ is the mountain pass value defined in (\ref{Fc}).
\end{enumerate}
\end{asu}
 Assumption \ref{asu-K} is more technical and as we show in Lemma \ref{LEMMA-3m} it is satisfied if $\Fc \not \in \{(2j+1)\Fm:j \in \N^*\}$, or more particularly if $\Fc < 3\Fm$. An analogous assumption was made by Bisgard in \cite{bisgard} with the same purpose. The comparison is made in subsection \ref{SUBS-previous}. Our first result then is as follows:
\begin{theorem}\label{THEOREM-general}
Assume that \ref{asu-sigma}, \ref{asu-infinity}, \ref{asu-zeros}, \ref{asu-heteroclinics}, \ref{asu-otherwells} and \ref{asu-K} hold. Then, there exists $\Fu \in H^1_{\loc}(\R,\R^k) \cap \CC^2(\R,\R^k)$ a solution of (\ref{connections-eq}) that satisfies one of the two following conditions:
\begin{enumerate}
\item $\Fu$ is not constant, $E(\Fu) \leq \Fc$ and $\Fu$ is \textit{homoclinic} to $\sigma^-$ or $\sigma^+$, that is, there exists $\sigma \in \{ \sigma^-,\sigma^+\}$ such that
\begin{equation}
\lim_{t \pm \infty} \Fu(t) = \sigma.
\end{equation}
\item $\Fu \in X(\sigma^-,\sigma^+)$ and $\Fc \geq E(\Fu) > \Fm$.
\end{enumerate}
Moreover, $\Fu(0) \in K$.
\end{theorem}
That is, Theorem \ref{THEOREM-general} shows that, under the previous assumptions, there exists a non-minimizing solution which might be either heteroclinic or homoclinic. As it will be made clear later, Theorem \ref{THEOREM-general} is strongly related to \textit{Theorem 2.3} by Bisgard \cite{bisgard}.
\begin{remark}\label{REMARK-eta_min}
Following the arguments by Bisgard \cite{bisgard} which give rise to his \textit{Theorem 2.2}, we also have that there exists a (possibly small) constant $\eta_{\min}>0$ such that if $\Fc<\Fm+\eta_{\min}$ ($\eta_{\min}<\Fm$), then $\Fu$ is heteroclinic and $J(\Fu)=\Fc$. See Corollary \ref{COROLLARY_existence_principle}.
\end{remark}
\subsubsection{The symmetric case}
In \cite{alama-bronsard-gui}, Alama, Bronsard and Gui considered potentials which are symmetric with respect to a reflection:
\begin{asu}\label{asu-symmetry}
We have that $\sigma^-= (-1,0,\ldots,0)$ and $\sigma^+=(+1,0,\ldots,0)$. Moreover, we have for all $\bu \in \R^k$, $V(\Fs(\bu))=V(\bu)$, where
\begin{equation}
\Fs: \bu=(u_1, u_2, \ldots, u_k) \in \R^k \to (-u_1, u_2, \ldots, u_k) \in \R^k.
\end{equation}
\end{asu}
Such condition eliminates the degeneracy due to invarance by translations and, hence, allows to restore some compactness. The first remark is that condition \ref{asu-symmetry} allows to look for solutions which belong to the equivariant space:
\begin{equation}
X_{\mathrm{sym},+}:= \{ q \in X(\sigma^-,\sigma^+): \forall t \geq 0, \hspace{1mm} q_1(t)\geq 0 \mbox{ and } \Fs(q(t))=q(-t)  \}.
\end{equation}
The purpose of the symmetry assumption \ref{asu-symmetry} is to replace  \ref{asu-K} in order to obtain a slightly better result. Moreover, we show that the combination of both hypothesis permits to ensure the existence of a non-minimizing heteroclinic in $X(\sigma^-,\sigma^+)$, while the general setting of Theorem \ref{THEOREM-general} does not allow us to claim such a thing (see however Remark \ref{REMARK-eta_min}). Firstly, we recall that assumption \ref{asu-symmetry} shows that energy decreases by symmetrization, see Lemma \ref{Lemma-Xsym} later. Therefore, we have that the sets
\begin{equation}\label{CFsym}
\forall i \in \{0,1\}, \hspace{2mm} \CF_{\sym,i}:=\CF_{i} \cap X_{\sym,+}
\end{equation}
are non-empty by \ref{asu-heteroclinics}. Moreover, $d(\CF_{\sym,0},\CF_{\sym,1}) \geq d(\CF_0,\CF_1) > 0$, again by assumption \ref{asu-heteroclinics}. We write $\CF_\sym:=\CF \cap X_{\sym,+}$, notice that $\CF_\sym=\CF_{\sym,0}\cup \CF_{\sym,1}$. We see that the function $\psi$ defined in (\ref{psi}) belongs to $X_{\sym,+}$. Hence, we can do as before and define:
\begin{equation}\label{H_sym}
\scrH_{\sym}:= \{ v \in \scrH: \forall t \geq 0, \hspace{1mm} \Fs(v(t))=v(-t) \},
\end{equation}
which is a closed subspace of $\scrH$, thus we will regard it as a Hilbert space itself. Notice that by Lemma \ref{LEMMA-difference} and the linearity of the symmetry, we have as before that $\{\psi\}+\scrH_\sym=X_\sym$. We set $\CV_\sym := \CF_\sym-\{\psi\}$ and for $i \in \{0,1\}$
\begin{equation}\label{CVsymi}
\CV_{\sym,i}:= \CF_{\sym,i}-\{\psi\},
\end{equation}
which are subsets of $\scrH$.  We now have all the ingredients to define the \textit{symmetric mountain pass family}
\begin{equation}\label{Gamma-sym}
\Gamma_\sym:= \{ \gamma \in C([0,1],\scrH_\sym): \forall i \in \{0,1\}, \gamma(i) \in \CV_{\sym,i}\}.
\end{equation}
As we will see later, the possibility of considering only the paths contained in $\scrH_\sym$ will be the key of our argument. Now, define the corresponding mountain pass value:
\begin{equation}\label{Fc-sym}
\Fc_\sym:= \inf_{\gamma \in \Gamma_\sym} \max_{s \in [0,1]}J(\gamma(s)) < +\infty.
\end{equation}
As before, we show that $\Fc_\sym>\Fm$ (Proposition \ref{PROPOSITION-sym-MP}). Subsequently, we write the analogous of \ref{asu-otherwells} for $\Fc_\sym$:
\begin{asu}\label{asu-otherwells-sym}
It holds that $\Fc_\sym < \Fm^\star$, where $\Fm^\star$ is introduced in (\ref{Fm_star}).
\end{asu}
We can finally state the first result in the symmetric setting:
\begin{theorem}\label{THEOREM-symmetry}
Assume that \ref{asu-sigma}, \ref{asu-infinity}, \ref{asu-zeros}, \ref{asu-heteroclinics}, \ref{asu-symmetry} and \ref{asu-otherwells-sym} hold. Then, we have one of the two following scenarios:
\begin{enumerate}
\item There exist $\Fu_+$ and $\Fu_-$ in $H^1_{\loc}(\R,\R^k) \cap \CC^2(\R,\R^k)$ two non constant functions such that $E(\Fu_+) \leq \Fc_\sym$, $\lim_{t \to \pm \infty}\Fu_+(t)=\sigma^+$ and $\Fu_-$ is obtained by reflecting $\Fu_+$, that is
\begin{equation}
\forall t \in \R, \hspace{2mm} \Fu_-(t)=\Fs(\Fu_+(t)).
\end{equation}
In particular, $\lim_{t \to \pm \infty} \Fu_-(t)=\sigma^-$.
\item There exists $\Fu \in X_{\sym,+}$ such that $E(\Fu)=\Fc_\sym$. In particular, $\Fu \not \in \CF_\sym$.
\end{enumerate}
\end{theorem}
\begin{remark}
Notice that in the first case in Theorem \ref{THEOREM-symmetry}, the solution $\Fu_-$ is obtained for free from $\Fu_+$. Indeed, it suffices to check that, due to \ref{asu-symmetry}, any $\Fq$ solution of (\ref{connections-eq}) gives rise to a reflected solution $\hat{\Fq}$ defined as $\hat{\Fq}: t \in \R \to \Fs(\Fq(t))$.
\end{remark}
Finally, we show that under an assumption which combines \ref{asu-K} and \ref{asu-symmetry} we can be sure to obtain a non-minimizing heteroclinic joining $\sigma^-$ and $\sigma^+$. Such assumption writes as follows:
\begin{asu}\label{asu-finalK}
Assumption \ref{asu-symmetry} holds. Moreover, there exists a closed set $K_\sym \subset \R^k$ such that:
\begin{enumerate}
\item There exists $\nu_0>0$ such that 
\begin{equation}
\forall \Fq \in \CF_\sym, \hspace{2mm} \mathrm{dist}(\Fq(0),K_\sym) \geq \nu_0
\end{equation}
where $\mathrm{dist}$ stands for the usual Euclidean distance between two sets in $\R^k$.
\item Let $\Gamma_\sym$ be as in (\ref{Gamma-sym}) and $\Fc_\sym$ be as in (\ref{Fc-sym}). There exists $M > \Fc_\sym$ such that for any $\gamma_+ \in \Gamma_\sym$ with $(\psi+\gamma_+)([0,1]) \subset X_{\sym,+}$ and $\max_{s \in [0,1]}J(\gamma_+(s)) \leq M$,
there exists $s_\gamma \in [0,1]$ such that $J(\gamma_+(s_\gamma)) \geq \Fc_\sym$ and $\gamma_+(s_\gamma)(0) \in K_\sym$.
\end{enumerate}
\end{asu}
Assumption \ref{asu-finalK} is nothing but the symmetric version of \ref{asu-K}. Notice that we also need to ask that $\gamma(s_\gamma)(0) \in K_\sym$, which is stronger than the condition $(\gamma(s_\gamma)+\psi)(\R) \cap K \not = \emptyset$ required in \ref{asu-K}. We can then state the following result:
\begin{theorem}\label{THEOREM-symmetry-final}
Assume that \ref{asu-sigma}, \ref{asu-infinity}, \ref{asu-zeros}, \ref{asu-heteroclinics},  \ref{asu-otherwells-sym} and \ref{asu-finalK} hold. Then, there exists a solution $\Fu \in X_{\sym,+}$ such that $\Fc_\sym \geq E(\Fu) > \Fm$.
\end{theorem}
\begin{remark}
Notice that the hypothesis of Theorem \ref{THEOREM-symmetry} are contained in those of Theorem \ref{THEOREM-symmetry-final}. Therefore, if $\Fu$ is the solution given by Theorem \ref{THEOREM-symmetry-final}, then by Theorem \ref{THEOREM-symmetry} either $E(\Fu)=\Fc_\sym$ or there exist $\Fu_+, \Fu_-$ the pair of non constant homoclinics.
\end{remark}

The paper is organized as follows: Section \ref{section-sketch} is devoted to the proofs of Theorems \ref{THEOREM-general}, \ref{THEOREM-symmetry} and \ref{THEOREM-symmetry-final}. Section \ref{section-asu} is devoted to some comments and results regarding the assumptions \ref{asu-K} and \ref{asu-finalK}.
\section{Proofs of the results}\label{section-sketch}
This section is devoted to the proofs of Theorems \ref{THEOREM-general}, \ref{THEOREM-symmetry} and \ref{THEOREM-symmetry-final}. The organisation goes as follows: In subsection \ref{SUBS-previous}, we give the overall scheme of the proofs and compare it with the previous literature. In subsection \ref{SUBS-Preliminary}, we state the preliminary results which are needed, most of which are well-known. In subsection \ref{SUBS-MP}, we prove the existence of the mountain pass geometry. In subsection \ref{SUBS-deformation}, we state an abstract deformation result from Willem \cite{willem} which is used after. In subsection \ref{SUBS-general}, we provide the proof of Theorem \ref{THEOREM-general}. Finally, subsection \ref{SUBS-SYM} is devoted to the proofs of Theorems \ref{THEOREM-symmetry} and \ref{THEOREM-symmetry-final}.

\begin{figure}[h!]
\centering
\includegraphics[scale=0.85]{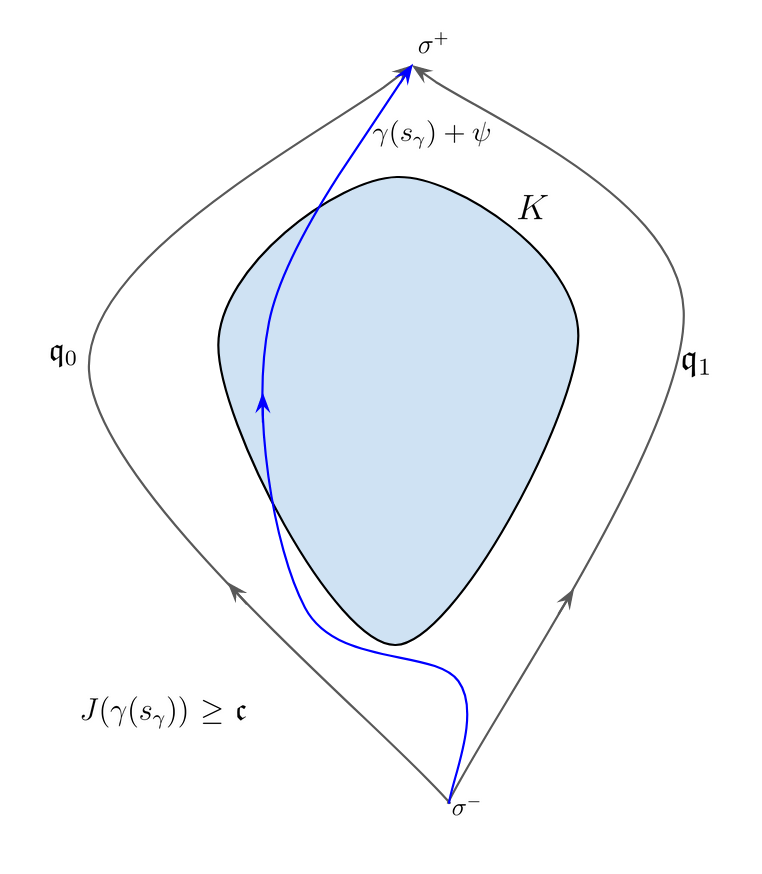}
\caption{Illustration of \ref{asu-K} for the particular case of a potential with exactly two distinct (up to translations) globally minimizing heteroclinics ($\Fq_0$ and $\Fq_1$) between $\sigma^-$ and $\sigma^+$.}
\end{figure}

\begin{figure}[h!]
\centering
\includegraphics[scale=0.4]{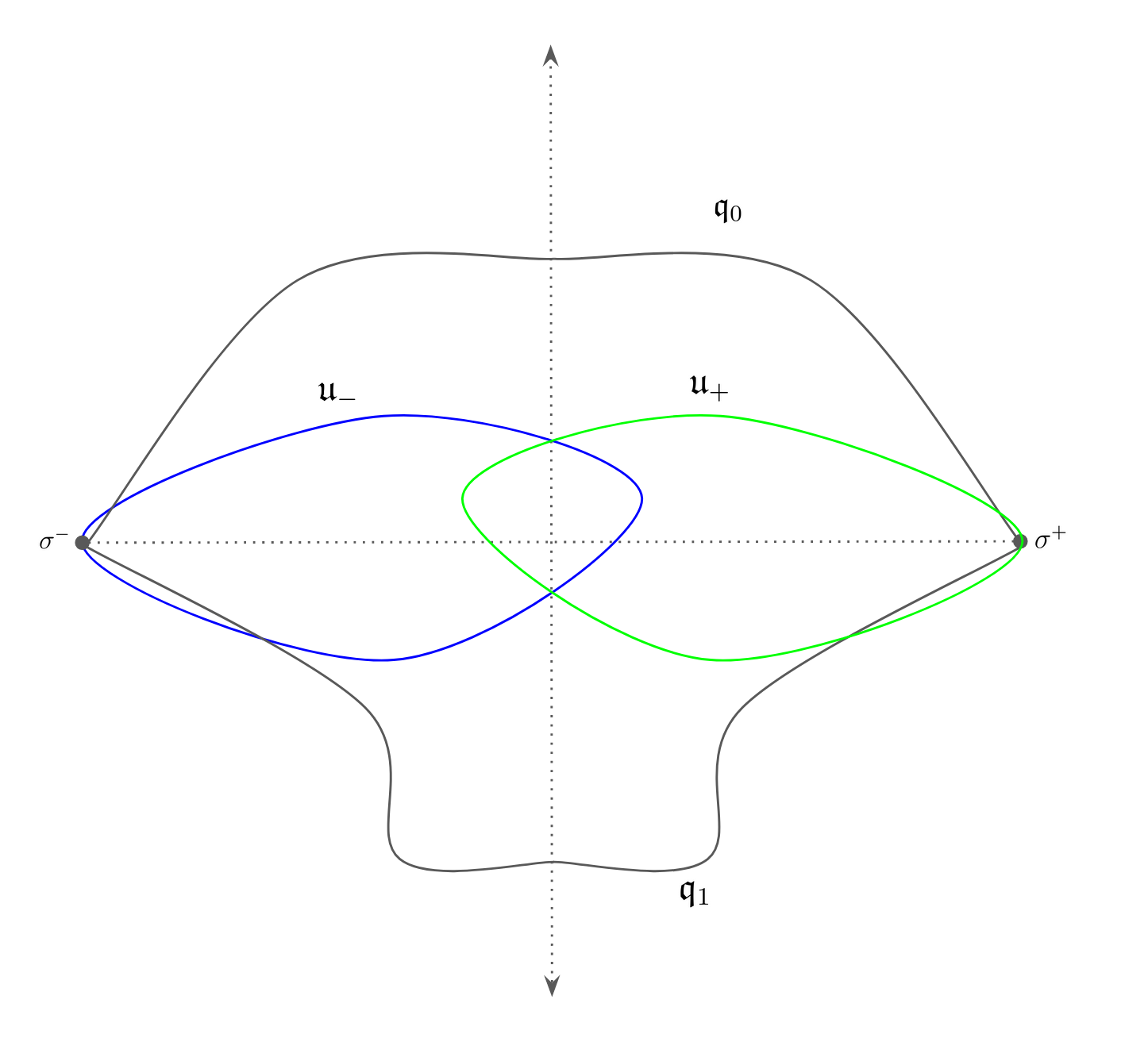}
\caption{Illustration of one of the two possible scenarios described in Theorem \ref{THEOREM-symmetry}. The heteroclinic orbits $\Fq_0$ and $\Fq_1$ represent two distinct symmetric minimizing heteroclinics. The two homoclinic orbits $\Fu_-$ and $\Fu_+$ are related by the reflection $\Fs$.}
\end{figure}
\subsection{Scheme of the proofs and comparison with the previous literature}\label{SUBS-previous}
As stated in the introduction, it is worth recalling that the problem of the existence of homoclinic and heteroclinic solutions for the second-order system of ODEs
\begin{equation}\label{V_*_eq}
q''(t)=\nabla_\bu V_{\star}(t,q(t)), \hspace{2mm} \forall t \in \R,
\end{equation}
using variational methods has been extensively studied during the past decades. In (\ref{V_*_eq}), $V_\star: \R \times \R^k \to \R$ is the potential, usually $T$-periodic in time. Some examples of early papers which use a mountain pass approach to find such solutions are Caldiroli and Montecchiari \cite{caldiroli-montecchiari}, Coti Zelati and Rabinowitz \cite{coti zelati-rabinowitz} and Rabinowitz \cite{rabinowitz90} (where the autonomous case is also treated). In those papers, the potential considered is quite far from being of multi-well type, meaning that the geometry of the associated functional is substantially different to the one considered in the present paper. On the contrary, in the papers Montecchiari and Rabinowitz \cite{montecchiari-rabinowitz18,montecchiari-rabinowitz20} as well as Bisgard \cite{bisgard}, $T$-periodic multi-well potentials $V_\star$ (with explicit time dependence) are considered. In this paper, we prove results which are very close (but not included) to those in \cite{bisgard} following an equivalent scheme of proof. More precisely, we rely in the following natural approach (as for instance in the seminal paper by Brézis and Nirenberg \cite{brezis-nirenberg}):
\begin{enumerate}
\item We prove the existence of a min-max value. In our case, we show in Proposition \ref{PROPOSITION-mountain-pass} that there exists a mountain pass value using the gap condition \ref{asu-heteroclinics}. The same is shown in Proposition \ref{PROPOSITION-sym-MP} for the symmetric setting
\item We analyze the behavior of the associated \textit{Palais-Smale sequences} in order to establish the existence of non-minimizing solutions from this analysis. This is the purpose of assumptions \ref{asu-K}, \ref{asu-symmetry} and \ref{asu-finalK}, which give rise to Theorems \ref{THEOREM-general}, \ref{THEOREM-symmetry} and \ref{THEOREM-symmetry-final} respectively.
\end{enumerate}
We now detail the previous steps of the proof and compare with \cite{bisgard}.
\subsubsection{The mountain pass geometry}
In order to obtain a mountain pass geometry, Bisgard and the other authors consider $\Fq$ a globally minimizing heteroclinic joining two wells $\sigma^-$ and $\sigma^+$. If $V_{\star}$ is say 1-periodic in time and the set $\Sigma:= \{ V_\star=0\}$ is $t$-independent, this implies that for any $n \in \Z$, $\Fq(\cdot+n)$ is also a globally minimizing heteroclinic. In order to establish the mountain pass geometry, Bisgard and the other authors define the family of paths
\begin{equation}
\Gamma_\star:= \{ \gamma \in C([0,1],\scrH): \gamma(0)=\Fq-\psi \mbox{ and } \gamma(1)=\Fq(\cdot+1) - \psi\},
\end{equation}
where $\psi$ is an interpolating function between $\sigma^-$ and $\sigma^+$ as in (\ref{psi}). If one considers the min-max value
\begin{equation}
\Fc_\star:= \inf_{\gamma \in \Gamma_\star}\max_{s \in [0,1]} J_\star(\gamma(s)),
\end{equation}
where
\begin{equation}
J_\star: v \in \scrH \to \int_{\R} \left[ \frac{\lvert v'(t)+\psi'(t) \rvert^2}{2}+V_\star(t,v(t)+\psi(t)) \right] dt,
\end{equation}
then $\Fc_\star>J_\star(\Fq-\psi)=J_\star(\Fq(\cdot+1)-\psi)$ implies
\begin{equation}\label{bisgard_asu}
\{ \Fq(\cdot+\tau): \tau \in [0,1]\} \mbox{ is not a continuum of globally minimizing heteroclinics,}
\end{equation}
see \textit{Proposition 2.1} in \cite{bisgard}. Since (\ref{bisgard_asu}) is never fulfilled if $V_\star$ is autonomous due to translation invariance, autonomous potentials are excluded from Bisgard's approach. Hence, in order to find a mountain pass value of this type for the case of autonomous potentials (that is, for the functional $J$ defined in (\ref{FunctionalJ})), we need then to add an additional assumption which produces a mountain pass geometry by playing a role analogous to (\ref{bisgard_asu}). As explained before, we do so by considering the natural candidate \ref{asu-heteroclinics} introduced in \cite{alessio}. Indeed, in Proposition \ref{PROPOSITION-mountain-pass} we show that such an assumption implies the existence of a mountain pass geometry for the autonomous case. Notice that (\ref{bisgard_asu}) only requires an explicit time dependence on the potential and, therefore, it does not exclude the scalar case. On the contrary, assumption \ref{asu-heteroclinics} for the autonomous problem is more restrictive and completely rules out scalar potentials.
\subsubsection{The analysis of the Palais-Smale sequences}
Once the mountain pass geometry has been established, the next natural step is to analyze the behavior of the Palais-Smale sequences at the mountain pass level, as the classical Palais-Smale condition is not satisfied by $J$ nor $J_\star$. For $J_\star$, this analysis is known and it can be found in \textit{Proposition 3.10} in Rabinowitz \cite{rabinowitz93}, as well as the results in Bisgard \cite{bisgard}, especially \textit{Theorem 1.21}. Condition (\ref{bisgard_asu}) is not necessary for proving those results, meaning that, in particular, they apply to our $J$, see Proposition \ref{PROPOSITION_ASYMPTOTIC_PS}. From this analysis it follows that Palais-Smale sequences (both for $J_\star$ and $J$) split into a chain of connecting orbits solving (\ref{connections-eq}) and that the sum of the energies of the elements of the chain is equal to the level of the Palais-Smale sequence. Using \ref{asu-otherwells}, we find that if one of the elements of the chain is not a globally minimizing heteroclinic between $\sigma^-$ and $\sigma^+$, then \textit{Theorem 2.3} in \cite{bisgard} or Theorem \ref{THEOREM-general} here is established. Nevertheless, there is still the possibility that each element of the limiting chain is a globally minimizing heteroclinic joining $\sigma^-$ and $\sigma^+$. In such a case, no new solution is produced by the mountain pass argument. Therefore, one needs to rule out this possibility by examining more closely the behavior of the Palais-Smale sequences at the mountain pass level. In \cite{bisgard}, the possibility of a chain of minimizing heteroclinics is excluded by imposing an assumption on the mountain pass level $c_\star$. More precisely, by setting
\begin{equation}\label{m_star_pm}
\Fm_\star^\pm:= \inf_{v \in \scrH}J_\star(v+\psi(\pm \cdot))
\end{equation}
if we have that
\begin{equation}\label{bisgard_asu2}
\Fc_\star \not \in \{ k_1 \Fm_\star^-+k_2 \Fm_\star^+: k_1+k_2 =2j+1, j \in \N^*\},
\end{equation}
then one of the elements of the limiting chain satisfies the requirements. This is essentially the assumption imposed by Bisgard in \textit{Theorem 2.3} \cite{bisgard}. In our case, assumption \ref{asu-K} serves the same purpose. The difference is that our argument is slightly more involved, as \ref{asu-K} does not allow to claim the desired conclusion in such a direct fashion. Instead, we show by a deformation procedure\footnote{We owe this idea to the referee. In previous versions of this paper, we relied instead on a lengthier and less direct argument based on a localized version of the mountain pass lemma due to Ghoussoub and Preis \cite{ghoussoub-preiss,ghoussoub}.} based on a result by Willem \cite{willem} that \ref{asu-K} implies that there exists a Palais-Smale sequence at the mountain pass level for which each element of the sequence goes through the set $K$, so it cannot be asymptotic to a formal chain of globally minimizing heteroclinics. The purpose of this approach is the following: as we show in Lemma \ref{LEMMA-3m}, if $\Fc$ satisfies
\begin{equation}\label{rel_Fc_K}
\Fc \not \in \{ (2j+1) \Fm: j \in \N^+\}
\end{equation}
then \ref{asu-K} holds. Relation (\ref{rel_Fc_K}) is nothing but the reformulation of (\ref{bisgard_asu2}) for the autonomous case. Indeed, in the autonomous setting, the values $\Fm^\pm_\star$ defined in (\ref{m_star_pm}) coincide (while they do not necessarily do in the non-autonomous case) meaning that (\ref{bisgard_asu2}) and (\ref{rel_Fc_K}) are the same. Therefore, one could assume (\ref{rel_Fc_K}) instead of \ref{asu-K} and obtain Theorem \ref{THEOREM-general} by the same way that in \cite{bisgard}. Nevertheless, as shown in Lemma \ref{LEMMA-3m} we have that \ref{asu-K} can be more general, so we worked under it instead of (\ref{rel_Fc_K}). In particular, the possibility $\Fc \in \{ (2j+1) \Fm: j \in \N^+\}$ is not excluded by \ref{asu-K}. We think that this feature is relevant as some addition phenomenon among the energies of several non-minimizing solutions in the chain could happen so that the total sum of the energies would be in $(2j+1)\N$. In this case, (\ref{bisgard_asu2}) would not allow to conclude while \ref{asu-K} would.

Another assumption is made by Bisgard in \cite{bisgard}, which leads to the stronger result \textit{Theorem 2.2}, where existence of an heteroclinic at the mountain pass level is shown. It consists on supposing that the mountain pass value is close enough to the minimum. The proof follows from the fact that for a range of values close enough to the minimum, no splitting on the Palais-Smale sequences can occur, meaning that they converge strongly. As we pointed out in Remark \ref{REMARK-eta_min}, the same result holds for our problem. The precise statement is given in Corollary \ref{COROLLARY_existence_principle}.

In any case, all the assumptions discussed before can be difficult to verify in applications. For this reason, we consider the more explicit symmetry assumption  \ref{asu-symmetry} in order to remove the degeneracy due to invariance by translations and recover some compactness. Under this assumption and \ref{asu-otherwells-sym}, we show that if we have dichotomy of the Palais-Smale sequence (which can be chosen such that it belongs to the appropriate symmetrized space $X_{\sym,+}$) then there exists a pair of non-constant homoclinic solutions. Theorem \ref{THEOREM-symmetry} is then deduced. The idea of using the symmetries in order to recover compactness and subsequently establishing existence and multiplicity results has been extensively used in the previous research, we refer for instance to the seminal paper by Berestycki and Lions \cite{berestycki-lions} as well as Van Schaftingen \cite{van schaftingen05} which contains some of the key ideas that we use in our approach and other material. Assumption \ref{asu-symmetry} has the advantage of being more explicit than \ref{asu-K}, (\ref{bisgard_asu2}) and (\ref{rel_Fc_K}), but it rules out a wide class of interesting non-symmetric potentials. We can also combine \ref{asu-symmetry} with \ref{asu-finalK}, which is the symmetrized version of \ref{asu-finalK}, in order to show the existence of a non-minimizing heteroclinic, which is Theorem \ref{THEOREM-symmetry-final}. This is done by relying again on the deformation argument.
\subsection{Preliminary results}\label{SUBS-Preliminary}
In this subsection, we state the technical preliminary results which will be used for establishing the main Theorems. They are for the most part essentially known and a few others are proven by classical arguments. Some relevant references which contain them (or close versions of them) are Rabinowitz \cite{rabinowitz93}, Bisgard \cite{bisgard}, Montecchiari and Rabinowitz \cite{montecchiari-rabinowitz18}, Bertotti and Montecchiari \cite{bertotti-montecchiari}, Alama, Bronsard and Gui \cite{alama-bronsard-gui}, Bronsard, Gui and Schatzman \cite{bronsard-gui-schatzman}. In several cases, we take results from those references and we rephrase them in order to be coherent with our setting. 

We being by recalling some basic properties on the potential $V$. These properties are easy to prove and well known, so the proofs are skipped. We refer, for instance, to \cite{bisgard} and see also \cite{alama-bronsard} for a particularization to the autonomous case. We first recall the following:
\begin{lemma}\label{LEMMA-difference}
Assume that \ref{asu-sigma} and \ref{asu-zeros} hold. Let $(\sigma_i,\sigma_j) \in \Sigma^2$. Let $q$ and $\tilde{q}$ be two elements in $X(\sigma_i,\sigma_j)$. Then $q -\tilde{q} \in \scrH$. Similarly, if $q \in X(\sigma_i,\sigma_j)$ and $v \in \scrH$ then $v+q \in X(\sigma_i,\sigma_j)$.
\end{lemma}
We refer for instance to \textit{Lemma 1.4} in \cite{bisgard} for a proof of this fact.
\begin{lemma}\label{LEMMA-nondegeneracy}
Assume that \ref{asu-sigma} and \ref{asu-zeros} hold. Then, there exist two positive constants $\delta$ and $\beta$ such that for all $\sigma \in \Sigma$
\begin{equation}\label{nondegeneracyV}
\forall \bu \in B(\sigma,\delta), \hspace{2mm} \beta^{-1}V(\bu) \leq \lvert \bu-\sigma \rvert^2 \leq \beta V(\bu).
\end{equation}
and
\begin{equation}\label{nondegeneracyNABLAV}
\forall \bu \in B(\sigma,\delta),\hspace{2mm} \beta^{-1} \langle \nabla V(\bu), \bu-\sigma \rangle \leq \lvert \bu -\sigma \rvert^2 \leq \beta \langle \nabla V(\bu), \bu-\sigma \rangle. 
\end{equation}
\end{lemma}
The constants $\delta$ and $\beta$ will be fixed for the latter. 
\begin{lemma}\label{LEMMA-limits-V}
Assume that \ref{asu-sigma} and \ref{asu-zeros} hold. Let $q \in H^1_{\loc}(\R,\R^k)$ satisfy $E(q) < +\infty$. Then
\begin{equation}
\lim_{t \to \pm \infty}V(q(t)) = 0.
\end{equation}
\end{lemma}
In order to apply the mountain pass lemma, we need to show that $J$ is a $C^1$ functional. This is done in \cite{bisgard} and \cite{montecchiari-rabinowitz18}. Let $(\sigma_i,\sigma_j) \in \Sigma^2$, following \cite{bisgard}, take $\chi \in X(\sigma_i,\sigma_j)$ and define
\begin{equation}
J_\chi : v \in \scrH \to E(\chi+v)
\end{equation}
which is well-defined by Lemma \ref{LEMMA-difference}. Under these notations, we have that the functional $J$ defined in (\ref{FunctionalJ}) is $J=J_\psi$, with $\psi$ as in (\ref{psi}). 
\begin{lemma}\label{LEMMA-C1}
Assume that \ref{asu-sigma} and \ref{asu-zeros} hold. Then, we have:
\begin{enumerate}[i)]
\item For any $(\sigma_i,\sigma_j) \in \Sigma^2$ and $\chi \in X(\sigma_i,\sigma_j)$, $J_\chi$ is a $C^1$ functional on $\scrH$ with derivative:
\begin{equation}\label{DJ}
\forall v \in \scrH, \hspace{2mm} DJ_\chi(v):w \in \scrH \to \int_\R \left( \langle \chi'+v',w' \rangle + \langle \nabla V(\chi+v),w \rangle \right) \in \R.
\end{equation}
In particular, if \ref{asu-sigma} and \ref{asu-zeros} hold and $DJ_\chi(v)=0$ for $v \in \scrH$, then $v+\chi$ solves (\ref{connections-eq}).
\item $J$ is $C^1$ as a functional restricted to $\scrH_\sym$ and its differential is as in (\ref{DJ}) with the proper modifications. If, moreover, we add the symmetry assumption \ref{asu-symmetry} and $v \in \scrH_\sym$ is such that $DJ(v)=0$ in $\scrH_\sym$, then $v+\psi$ solves (\ref{connections-eq}). 
\end{enumerate}
\end{lemma}
Item i) in Lemma \ref{LEMMA-C1} is essentially \textit{Proposition 1.6}  in \cite{bisgard}, for the particular case of autonomous potentials. The proof of item ii) follows from classical arguments using assumption \ref{asu-symmetry}, so we skip it. Next, we recall the following general property for sequences with uniformly bounded energy:
\begin{lemma}\label{LEMMA-boundedE}
Assume that \ref{asu-sigma}, \ref{asu-infinity} and \ref{asu-zeros} hold. Let $(\sigma_i,\sigma_j) \in \Sigma^2$. Let $(q_n)_{n \in \N}$ be a sequence in $X(\sigma_i,\sigma_j)$ such that $\sup_{n \in \N}E(q_n) <+\infty$. Then, up to an extraction, there exists $q \in H^1_{\loc}(\R,\R^k)$ such that $q_n \to q$ locally uniformly and $q_n' \rightharpoonup q'$ weakly in $\scrL$. Moreover, $E(q) \leq \liminf_{n \to \infty} E(q_n)$.
\end{lemma}
The property given by Lemma \ref{LEMMA-boundedE} is certainly well-known and the proof is classical, so we omit it. As we see, a uniform bound on the energy is not sufficient to obtain control on the behavior of the sequence of infinity. This is due to the fact that $V$ possesses more than one zero and it is the cause of non-existence phenomena already when dealing with the minimization problem. Using Lemma \ref{LEMMA-boundedE}, we obtain by classical arguments the following property for arbitrary Palais-Smale sequences:
\begin{lemma}\label{LEMMA-PSbasic}
Assume that \ref{asu-sigma}, \ref{asu-infinity} and \ref{asu-zeros} hold. Let $(v_n)_{n \in \N}$ be Palais Smale sequence at $c \geq \Fm$, i.e.,
\begin{equation}\label{HYP-PSbasic}
\lim_{n \to \infty}J(v_n) = c \mbox{ and } \lim_{n \to \infty}DJ(v_n) = 0 \mbox{ in } \scrH.
\end{equation}
Then, the following holds:
\begin{enumerate}
\item There exists a subsequence of $(v_n)_{n \in \N}$ (not relabeled) and $q \in H^1_{\loc}(\R,\R^k)$ such that
\begin{equation}\label{CONV-PSbasic}
\forall S_K \subset \R \mbox{ compact,} \hspace{2mm} \psi+v_n \to_{n \to \infty} q \mbox{ strongly in } H^1(S_K,\R^k).
\end{equation}
Moreover, $E(q) \leq c$ and $q \in \CC^2(\R,\R^k)$ solves (\ref{connections-eq}).
\item For any $(\tau_n)_{n \in \N}$ a sequence of real numbers, the sequence $(v_n^{\tau_n})_{n \in \N}$ defined as
\begin{equation}
\forall n \in \N, \hspace{2mm} v_n^{\tau_n}:= \psi(\cdot+\tau_n)+v_n(\cdot+\tau_n)-\psi
\end{equation}
is a Palais-Smale sequence at the level $c$ as in (\ref{HYP-PSbasic}).
\end{enumerate}
\end{lemma}
\begin{proof}
We show the first part. Define $(q_n)_{n \in \N}:= (\psi+v_n)_{n \in \N}$, which is a sequence contained in $X(\sigma^-,\sigma^+)$. Using Lemma \ref{LEMMA-boundedE} and the first part of the Palais-Smale condition (\ref{HYP-PSbasic}), we find $q \in H^1_{\loc}(\R,\R^k)$ and a subsequence (not relabeled) such that $E(q) \leq c$, $q_n \to q$ locally uniformly and $q_n'\rightharpoonup q'$ in $\scrL$. We show the local convergence with respect to the $H^1$ norm. Let $S_K \subset \R$ be compact and $v_{S_K} \in \scrH$ with $\supp(v_K) \subset S_K$. Using Cauchy-Schwartz inequality, we have
\begin{align}
\left\lvert \int_\R\langle \nabla V(q_n)- \nabla V(q), v_K  \rangle \right\rvert &= \left\lvert \int_K \int_0^1\langle D^2V(\lambda q+(1-\lambda)q_n)(q_n-q)   ,v_K \rangle d \lambda \right\rvert \\ \label{PSbasic-expr1}
&\leq C_K\lVert q_n-q \rVert_{L^2(K,\R^k)}\lVert v_K \rVert_{L^2(K,\R^k)},
\end{align}
where $C_K:= \max_{K}(D^2V(\lambda q+(1-\lambda)q_n))$. We have that $q_n \to q$ uniformly in $S_K$, so due to the continuity of $D^2V$ we have that $C_K < \infty$ and $C_K$ independent on the sequence $(q_n)_{n \in \N}$. Using (\ref{DJ}) and (\ref{PSbasic-expr1}), we write
\begin{equation}\label{PSbasic-expr2}
\left\lvert \int_K \langle q_n'-q', v_K' \rangle \right\rvert \leq C_K\lVert q_n-q \rVert_{L^2(K,\R^k)}\lVert v_K \rVert_{L^2(K,\R^k)}+ DJ(v_n)(v_K).
\end{equation}
Taking the supremum in (\ref{PSbasic-expr2}) for $v_K \in \scrH$ with $\supp(v_K) \subset S_K$ and $\lVert v_K \rVert_{\scrH} \leq 1$, by the dual characterization of the norm of a Hilbert space we get
\begin{equation}\label{PSbasic-expr3}
\lVert q_n' -q' \rVert_{L^2(K,\R^k)} \leq C_K \lVert q_n-q \rVert_{L^2(K,\R^k)}+\lVert DJ(v_n) \rVert_{\scrH}.
\end{equation}
Since $q_n \to q$ uniformly in $K$, we have $q_n \to q$ in $L^2(S_K,\R^k)$. In addition, the Palais-Smale condition (\ref{HYP-PSbasic}) implies $\lVert DJ(v_n) \rVert_{\scrH} \to 0$. Therefore, we have 
\begin{equation}
\lVert q_n' - q' \rVert_{L^2(S_K,\R^k)} \to 0,
\end{equation}
meaning that $q_n \to q$ in $H^1(S_K,\R^k)$, as we wanted to show. It only remains to show that $q$ solves (\ref{connections-eq}). Take $\varphi \in \CC^\infty_c(\R,\R^k)$. The convergence of the sequence inside $H^1(\supp(\varphi),\R^k)$ is strong, meaning that we can show
\begin{equation}
DJ(v)(\varphi) = \lim_{n \to \infty} DJ(v_n)(\varphi) = 0.
\end{equation}
In conclusion
\begin{equation}
\forall \varphi \in \CC_c(\R,\R^k), \hspace{2mm} \int_\R \left[ \langle q',\varphi' \rangle + \langle \nabla V(q),\varphi \rangle  \right]=0,
\end{equation}
which by classical regularity arguments means that $q$ is a solution of (\ref{connections-eq}) which belongs to $\CC^2(\R,\R^k)$. 

For proving part 2, it suffices to write for any $n \in \N$ and $\varphi \in \scrH$
\begin{equation}
DJ(v_n^{\tau_n})(\varphi)=DJ(v_n)(\varphi(\cdot-\tau_n)),
\end{equation}
which by taking the supremum in the unit ball of $\scrH$ gives
\begin{equation}
\lVert DJ(v_n^{\tau_n})\rVert_\scrH = \lVert DJ(v_n) \rVert_\scrH.
\end{equation}
\qed
\end{proof}
As in Lemma \ref{LEMMA-boundedE}, Lemma \ref{LEMMA-PSbasic} gives no control on the convergence of the elements of the sequence at infinity. In particular, in general the functional $J$ \textit{does not} satisfy the so-called Palais-Smale condition\footnote{We say that $J$ satisfies the Palais-Smale condition at the level $c \geq \Fm$ if every sequence satisfying (\ref{HYP-PSbasic}) possesses a convergent subsequence in $\scrH$.}, at least for arbitrary $c \geq \Fm$. The problem is not fixed even if we use the translation invariance property from the second part of Lemma \ref{LEMMA-PSbasic}. As explained already, assumptions \ref{asu-K}, \ref{asu-symmetry} and \ref{asu-wells} are introduced in order to circumvent this issue. Assumptions \ref{asu-otherwells} and \ref{asu-otherwells-sym} are made in order to exclude the possibility that the Palais-Smale sequences at the mountain pass levels originate a globally minimizing connecting orbit joining a well in $\{ \sigma^-,\sigma^+\}$ and a well in $\Sigma \setminus \{ \sigma^-,\sigma^+\}$. This is shown by the following:
\begin{lemma}\label{LEMMA-3rdwell}
Assume that \ref{asu-sigma}, \ref{asu-infinity} and \ref{asu-zeros} hold. Let $q \in X(\sigma^-,\sigma^+)$ be such that
\begin{equation}\label{3rdwell-hyp}
E(q) \leq C,
\end{equation}
where $C<\Fm^\star$, where $\Fm^\star$ is as in (\ref{Fm_star}). There exists $\rho_2(C)>0$, depending only on $V$ and $C$, such that
\begin{equation}\label{3rdwell-conclusion}
\forall \sigma \in \Sigma \setminus \{ \sigma^-,\sigma^+\}, \forall t \in \R, \hspace{2mm} \lvert q(t)-\sigma \rvert \geq \rho_2(C).
\end{equation}
\end{lemma}
Lemma \ref{LEMMA-3rdwell} is a straightforward generalization of results which where known previously, see \cite{alama-bronsard-gui} and \cite{bronsard-gui-schatzman}. The proof is skipped.

We conclude this paragraph by recalling that the complete asymptotic analysis of the Palais-Smale sequences and some of the consequences that follow are available in \cite{rabinowitz93} and \cite{bisgard}. Such properties do not play a major role in our argument\footnote{Proposition \ref{PROPOSITION_ASYMPTOTIC_PS} is only invoked once, in the proof of Lemma \ref{LEMMA-3m} and Corollary \ref{COROLLARY_existence_principle} is brought into account in Remark \ref{REMARK-eta_min}} the reason being that we find Lemma \ref{LEMMA-PSbasic} is better adapted to our purposes. The main result can be stated as follows for our setting:
\begin{proposition}\label{PROPOSITION_ASYMPTOTIC_PS}
Assume that \ref{asu-sigma}, \ref{asu-infinity} and \ref{asu-zeros} hold. Let $(\sigma_i,\sigma_j) \in \Sigma^2$, $c \in \R$, $\chi \in X(\sigma_i,\sigma_j)$ and $(v_n)_{n \in \N}$ be a Palais-Smale sequence for $J_\chi$ at the level $c$. Then, up to an extraction there exists $j \in \N^*$, such that there is $(A_{n}^i)_{n \in \N, i\in \{ 1,\ldots,j \}}$ a sequence of adjacent sub-intervals of $\R$, $(\tau_{n}^i)_{n \in \N, i \in \{1,\ldots,j\} }$ a sequence of translates in $\R$ and $q^1, \ldots, q^j$ solutions of (\ref{connections-eq}) such that:
\begin{enumerate}
\item For all $n \in \N$, $\cup_{i =1}^jA_n^i=\R$.
\item For all $i \in \{ 1, \ldots,j-1\}$, we have 
\begin{equation}
\lim_{t \to -\infty} q^{i+1}(t)=\lim_{t \to +\infty} q^i(t).
\end{equation}
Moreover, 
\begin{equation}
\lim_{t \to -\infty}q^1(t)=\sigma_i \mbox{ and } \lim_{t \to +\infty}q^j(t)=\sigma_j.
\end{equation}
\item For all $i \in \{1,\ldots,j\}$ we have that
\begin{equation}
\lim_{n \to \infty} \lVert v_n+\chi-q^i(\cdot-\tau_n^i) \rVert_{H^1(A_n^i,\R^k)}=0
\end{equation}
\item For all $i \in \{ 1, \ldots,j-1\}$, it holds that $\tau_n^{i+1}-\tau_{n}^i \to +\infty$ as $n \to \infty$.
\item $c=\sum_{i=1}^j E(q^i)$.
\end{enumerate}
\end{proposition}
Proposition \ref{PROPOSITION_ASYMPTOTIC_PS} is essentially \textit{Proposition 3.10} by Rabinowitz \cite{rabinowitz93}, with the main difference that we do not restrict to double-well potentials and we particularize to the autonomous case. The modifications needed in order to adapt the proof in \cite{rabinowitz93} are minor, so we do not include them. Proposition \ref{PROPOSITION_ASYMPTOTIC_PS} can also be deduced from the results in \cite{bisgard}. As already explained, in \cite{bisgard} this analysis is used to obtain existence results for non-minimizing connecting orbits under an assumption on the mountain pass value. We briefly recall the procedure. We first recall the following property, which is equivalent to \textit{Corollary 1.18} in \cite{bisgard} and \textit{Lemma 3.6} in \cite{rabinowitz93}. It states that there exists an inferior bound depending only on $V$ for the energy of non-constant connecting orbits:
\begin{lemma}\label{LEMMA_sol_inf_bound}
Assume that \ref{asu-sigma}, \ref{asu-infinity} and \ref{asu-zeros} hold. There exists $\eta_{\min}>0$ such that for any $(\sigma_i,\sigma_j) \in \Sigma^2$, if $q \in X(\sigma_i,\sigma_j)$ solves (\ref{connections-eq}) then either $E(q) \geq \eta_{\min}$ or $q$ is constant.
\end{lemma}
The proof of Lemma \ref{LEMMA_sol_inf_bound} follows from the fact that $V$ is stricly convex in a neighbourhood of the wells. We refer to the references mentioned before for a proof. Inspecting the proof of those results, we see that $\eta_{\min}$ is of the order of $\delta$ from Lemma \ref{LEMMA-nondegeneracy}, which can be very small. Proposition \ref{PROPOSITION_ASYMPTOTIC_PS} and Lemma \ref{LEMMA_sol_inf_bound} can be combined in order to easily obtain the following existence principle, which is essentially the result by Bisgard:
\begin{corollary}\label{COROLLARY_existence_principle}
Assume that \ref{asu-sigma}, \ref{asu-infinity} and \ref{asu-zeros} hold. Let $(\sigma_i,\sigma_j) \in \Sigma^2$, $c \in \R$, $\chi \in X(\sigma_i,\sigma_j)$ and $(v_n)_{n \in \N}$ a Palais-Smale sequence for $J_\chi$ at the level $c$. Then, we have:
\begin{enumerate}[i)]
\item If $c<\Fm_{ij}+\eta_{\min}$, where $\Fm_{ij}$ is defined in (\ref{mij}) and $\eta_{\min}$ is the constant from Lemma \ref{LEMMA_sol_inf_bound}, then there exists $\Fq_c \in X(\sigma_i,\sigma_j)$ and a sequence of real numbers $(\tau_n)_{n \in \N}$ such that $v_n+\chi - \Fq_c(\cdot-\tau_n) \to 0$ strongly in $\scrH$ up to subsequences. In particular, $\Fq_c$ solves (\ref{connections-eq}) and $E(\Fv_c+\chi)=c$.
\item If $c \not \in \{ (2l+1) \Fm_{ij}: l \in \N^*\}$ there exists $\tilde{\Fu}_c$ a solution to (\ref{connections-eq}) which is not a globally minimizing connecting orbit joining $\sigma_i$ and $\sigma_j$.
\end{enumerate}
\end{corollary}
Up to the obvious minor modifications, i) in Corollary \ref{COROLLARY_existence_principle} corresponds to \textit{Theorem 2.2} in \cite{bisgard} and ii) is \textit{Theorem 2.3} in the same reference. While in \cite{bisgard} those results are particularized to $\sigma_i=\sigma^-$, $\sigma_j=\sigma^+$ and $c=\Fc$ as in (\ref{Fc}), an examination of the arguments shows that it also applies to the case $\sigma_i=\sigma_j$ and for any level $c$ possessing a Palais-Smale sequence, so there is no obstacle for this more general statement. Nevertheless, it is important to notice as we already did in Remark \ref{REMARK-eta_min} that by i) we have that if $\Fc<\Fm+\eta_{\min}$, then there exists a mountain pass heteroclinic in $X(\sigma^-,\sigma^+)$ with energy $\Fc$. The counterpart of this statement is that the value $\eta_{\min}$ can be very small, as we point out after the statement of Lemma \ref{LEMMA_sol_inf_bound}. Notice also that by combining Lemma \ref{LEMMA_sol_inf_bound} and i) in Corollary \ref{COROLLARY_existence_principle} we have that for any $c \in (0,\eta_{\min})$ there is not any Palais-Smale sequence for $J_\chi$ at the level $c$, where $\chi \in X(\sigma,\sigma)$ and $\sigma \in \Sigma$. 
\subsection{Existence of a mountain pass geometry}\label{SUBS-MP}
The existence of a mountain  pass geometry is proven by combining \ref{asu-heteroclinics} with the last part of the following well-known result:

\begin{theorembis}\label{THEOREM-0}
Assume that \ref{asu-sigma}, \ref{asu-infinity}, \ref{asu-zeros} and \ref{asu-wells} hold. Then, there exists $\Fq \in X(\sigma^-,\sigma^+)$ such that $E(\Fq)=\Fm$, where $\Fm$ is as in (\ref{Fm}). Moreover, if $(q_n)_{n \in \N}$ is a minimizing sequence in $X(\sigma^-,\sigma^+)$, there exists a subsequence (not relabeled) and a sequence $(\tau_n)_{n \in \N}$ of real numbers such that $q_n(\cdot+\tau_n)-\tilde{\Fq} \to 0$ strongly in $\scrH$, for some $\tilde{\Fq} \in X(\sigma^-,\sigma^+)$ such that $E(\tilde{\Fq})=\Fm$.
\end{theorembis}
The existence part in Theorem \ref{THEOREM-0}, under different forms but using analogous arguments, can be found in several references. See for instance Bolotin  \cite{bolotin}, Bolotin and Kozlov \cite{bolotin-kozlov}, Bertotti and Montecchiari \cite{bertotti-montecchiari} and Rabinowitz \cite{rabinowitz89,rabinowitz92}. Proofs which use other type of arguments can be also found in Alikakos and Fusco \cite{alikakos-fusco}, Monteil and Santambrogio \cite{monteil-santambrogio}, Zuñiga and Sternberg \cite{zuniga-sternberg}. Regarding the compactness of the minimizing sequences and the applications of this property to some PDE problems, see Alama, Bronsard and Gui \cite{alama-bronsard-gui}, Alama et. al. \cite{alama-bronsard} and Schatzman \cite{schatzman}. As it is well known, \ref{asu-wells} might not be necessary but it cannot be removed, see Alikakos, Betelú and Chen \cite{alikakos-betelu-chen} for some counterexamples. We can now establish the existence of a mountain pass geometry:

\begin{proposition}\label{PROPOSITION-mountain-pass}
Assume that \ref{asu-sigma}, \ref{asu-infinity}, \ref{asu-zeros} and \ref{asu-heteroclinics} hold. Let $\Fc$ be as in (\ref{Fc}). Then, we have $\Fc > \Fm$.
\end{proposition}
\begin{proof}
Let $\gamma \in \Gamma$. By \ref{asu-heteroclinics} and using the definition of $\CV_0$ and $\CV_1$, we have that
\begin{equation}
\rho:= \mathrm{dist}_\scrH(\CV_0,\CV_1)=d(\CF_0,\CF_1)>0,
\end{equation}
where $\mathrm{dist}_\scrH$ denotes the distance between two sets in $\scrH$. Since $\CV_0 \cup \CV_1 = \CV$ and $\gamma$ is a continuous path which joins $\CV_0$ and $\CV_1$, we have that there exists $s^\star \in [0,1]$ such that
\begin{equation}\label{s_star}
\mathrm{dist}_\scrH(\gamma(s^\star),\CV) \geq \frac{\rho}{4}.
\end{equation}
We claim that there exists $c(\rho)>0$ such that for all $v \in \scrH$ verifying
\begin{equation}
\mathrm{dist}(v,\CV)\geq \frac{\rho}{4}
\end{equation}
we have $J(v) \geq \Fm+c(\rho)$. This is actually a well know result (see \cite{alama-bronsard,schatzman}), which is a straightforward consequence of the compactness property for minimizing sequences given by Theorem \ref{THEOREM-0}. Thus, by (\ref{s_star}) we obtain $\Fc \geq \Fm +c(\rho)$, which concludes the proof.\qed
\end{proof}

Subsequently, we establish the existence of a mountain pass geometry under the symmetry assumption. We begin by the following preliminary result:

\begin{lemma}\label{Lemma-Xsym}
Assume that \ref{asu-sigma}, \ref{asu-infinity}, \ref{asu-zeros} and \ref{asu-symmetry} hold. Let $q \in X(\sigma^-,\sigma^+)$. Then, there exist $q_{\mathrm{sym}} \in X_{\mathrm{sym}}$ and $q_{\mathrm{sym},+} \in X_{\mathrm{sym},+}$ such that we have
\begin{equation}
E(q_{\mathrm{sym},+})\leq E(q_{\mathrm{sym}}) \leq E(q).
\end{equation}
\end{lemma}
\begin{proof}
Let $q \in X(\sigma^-,\sigma^+)$. By the intermediate value Theorem, there exists $\tau \in \R$ such that
\begin{equation}
q_1(\tau)=0. 
\end{equation}
Due to the translation invariance of the energy, we can assume that $\tau=0$ (otherwise, replace $q$ by $q(\cdot+\tau))$. Without loss of generality, assume that
\begin{equation}\label{Xsym-ineq-q}
\int_0^{+\infty}e(q) \leq \int_{-\infty}^0 e(q).
\end{equation}
We define $q_{\mathrm{sym}}$ as
\begin{equation}
q_{\mathrm{sym}}(t):= \begin{cases}
q(t) &\mbox{ if } t \geq 0,\\
\Fs(q(-t)) &\mbox{ if } t \leq 0,
\end{cases}
\end{equation}
which is well defined and belongs to $X_{\mathrm{sym}}$. Notice that, due to this last fact, assumption \ref{asu-symmetry} and (\ref{Xsym-ineq-q})
\begin{equation}
E(q_{\mathrm{sym}})=2\int_0^{+\infty} e(q) \leq E(q).
\end{equation}
Subsequently, we set
\begin{equation}
q_{\mathrm{sym},+}(t):=\begin{cases}
(\lvert (q_\mathrm{sym})_1(t) \rvert, (q_\mathrm{sym})_2(t),\ldots, (q_\mathrm{sym})_k(t)) &\mbox{ if } t \geq 0,\\
(-\lvert (q_\mathrm{sym})_1(t) \rvert, (q_\mathrm{sym})_2(t),\ldots, (q_\mathrm{sym})_k(t)) &\mbox{ if } t \leq 0.
\end{cases}
\end{equation}
The function $q_{\mathrm{sym},+}$ is also well defined and belongs to $X_{\mathrm{sym},+}$. By assumption \ref{asu-symmetry}, we have for all $t \in \R$ that $V(q_{\mathrm{sym},+}(t))=V(q_{\mathrm{sym}}(t))$ and, by definition, we also have $\lvert q_{\mathrm{sym},+}' \rvert \leq \lvert q_{\mathrm{sym}}' \rvert$, a.e. in $\R$. Therefore,
\begin{equation}
E(q_{\mathrm{sym},+})\leq E(q_{\mathrm{sym}}),
\end{equation}
which establishes the proof.
\qed
\end{proof}

\begin{proposition}\label{PROPOSITION-sym-MP}
Assume that \ref{asu-sigma}, \ref{asu-infinity}, \ref{asu-zeros}, \ref{asu-heteroclinics} and \ref{asu-symmetry} hold. Let $\Fc_\sym$ be as in (\ref{Fc-sym}). Then, we have $\Fc_\sym > \Fm$.
\end{proposition}
\begin{proof}
We have the following result which shows that coercivity also holds in the equivariant setting (see \cite{alama-bronsard-gui} for a proof):
\begin{lemma}[Alama-Bronsard-Gui \cite{alama-bronsard-gui}, Lemma 2.4]\label{LEMMA-ABG}
For any $\eps>0$, there exists $c(\eps)>0$ such that for any $q \in X_\sym$ such that $E(q)<\Fm+c(\eps)$ we have $\lVert q-\Fq \rVert_{\scrH} < \eps$ for some $\Fq \in \CF_\sym$.
\end{lemma}
Using Lemma \ref{LEMMA-ABG} as well as Lemma \ref{Lemma-Xsym}, it suffices to apply the argument given in the proof of Proposition \ref{PROPOSITION-mountain-pass} to conclude.
\qed 
\end{proof}

Combining Lemma \ref{LEMMA-C1} and Proposition \ref{PROPOSITION-mountain-pass}, the classical mountain pass lemma states that there exists a Palais-Smale sequence at a level $\Fc$, i. e., a sequence $(v_n)_{n \in \N}$ in $\scrH$ such that
\begin{equation}\label{PS-Fc}
\lim_{n \to \infty} J(v_n) = \Fc \mbox{ and } \lim_{n \to \infty}DJ(v_n) = 0 \mbox{ in } \scrH.
\end{equation}
Similarly, by Proposition \ref{PROPOSITION-sym-MP} we find a sequence $(v_n')_{n \in \N}$ in $\scrH_\sym$ such that
\begin{equation}\label{PS-Fc-sym}
\lim_{n \to \infty} J(v_n') = \Fc_\sym \mbox{ and } \lim_{n \to \infty}DJ(v_n') = 0 \mbox{ in } \scrH_\sym.
\end{equation}

\subsection{An abstract deformation lemma}\label{SUBS-deformation}

As explained before, assumptions \ref{asu-K} and \ref{asu-finalK} are used in order to produce Palais Smale sequences at the mountain pass levels such that each element of the sequences goes through a suitable subset of $\R^k$. In order to show the existence of these sequences, we will use a deformation lemma due to Willem. Let us recall some standard terminology. Given a Banach space $X$ we denote by $X'$ its topological dual and given $I \in C^1(X)$, $DI$ is its derivative and for $c \in \R$, $I^c:=\{ x \in X:I(x) \leq c\}$. Given $S \subset X$ and $\rho>0$, we write $S_{\rho}:= \{ x \in X: \mathrm{dist}_X(x,S) \leq \rho\}$.  The result we will invoke is as follows:
\begin{lemma}[Willem, \textit{Lemma 2.3} \cite{willem}]\label{LEMMA_Willem}
Let $X$ be a Banach space, $I \in C^1(X)$, $S \subset X$, $c \in \R$, $\eps, \rho>0$ such that
\begin{equation}\label{deformation_contr}
\forall x \in I^{-1}([c-2\eps,c+2\eps]) \cap S_{2\rho}, \hspace{2mm} \lVert D I(x) \rVert_{X'} \geq 8\eps/\rho
\end{equation}
Then, there exists $\eta \in C([0,1]\times X,X)$ such that
\begin{enumerate}[(i)]
\item $\eta(t,u)=u$ if $t=0$ or if $u \not \in I^{-1}([c-2\eps,c+2\eps]) \cap S_{2\rho}$.
\item $\eta(1,I^{c+\eps} \cap S) \subset I^{c-\eps}$.
\item For all $t \in [0,1]$, $\eta(t,\cdot)$ is an homeomorphism of $X$.
\item For all $x \in X$ and $t \in [0,1]$, $\lVert \eta(t,x)-x \rVert_{X} \leq \delta$.
\item For all $x \in X$, $I(\eta(\cdot,x))$ is non increasing.
\item For all $x \in \varphi^{c} \cap S_{\rho}$ and $t \in (0,1]$, $I(\eta(t,u))<c$.
\end{enumerate}
\end{lemma}
Roughly speaking, the key point of Lemma \ref{LEMMA_Willem} is that if (\ref{deformation_contr}) holds then there exists a homotopy equivalence between $I^{c+\eps} \cap S$ and a subset of $I^{c-\eps}$. Equivalently, if we can find $S$ such that there is not any homotopy equivalence between $I^{c+\eps} \cap S$ and any $S' \subset I^{c-\eps}$, then (\ref{deformation_contr}) does not hold. The purpose of properties such as \ref{asu-K} or \ref{asu-finalK} is to provide such a set $S$.
\subsection{The proof of Theorem \ref{THEOREM-general}}\label{SUBS-general}
The idea of the proof of Theorem \ref{THEOREM-general} is to show the existence of a Palais-Smale sequence at the level $\Fc$ ($\Fc$ as in (\ref{Fc})) which produces a solution $\Fu$ such that $\Fu(0) \in K$, which is hence not in $\CF$. It is here when \ref{asu-K} enters. We define the set
\begin{equation}\label{setF}
F:=\{v \in \scrH: (v+\psi)(\R) \cap K \not = \emptyset\}
\end{equation}
with $K$ as in \ref{asu-K}. We show the following:
\begin{proposition}\label{PROPOSITION-local-PS}
There exists sequences, $(u_n)_{n \in \N}$ in $\scrH$ and $(\tau_n)_{n \in \N}$ in $\R$, such that
\begin{enumerate}
\item $J(u_n) \to \Fc$ as $n \to \infty$.
\item $DJ(u_n) \to 0$ in $\scrH$ as $n \to \infty$.
\item For all $n \in \N$, there exists $\tau_n \in \R$ 
\begin{equation}
\lim_{n \to \infty}\mathrm{dist}_\scrH(u_n(\tau_n)+\psi(\tau_n),F) = 0.
\end{equation}
\end{enumerate}
\end{proposition}
\begin{proof}
We prove the result by contradiction. If a sequence as in the statement does not exist, then we can find $h \in (0,\frac{1}{2}\min\{M-\Fc,\Fc-\Fm\})$ ($M$ as in \ref{asu-K}. Recall also that $\Fc>\Fm$ due to Proposition \ref{PROPOSITION-mountain-pass}), $\mu>0$ and $\nu>0$ such that
\begin{equation}\label{Deformation_J}
\forall v \in J^{-1}([c-h,c+h]) \cap F_{\nu}, \hspace{2mm} \lVert DJ(v) \rVert_{\scrH} \geq \mu
\end{equation}
with $F$ as in (\ref{setF}) and $F_{\nu}:=\{ v \in \scrH, \mathrm{dist}(v,F) \leq \nu\}$. We have that (\ref{Deformation_J}) is (\ref{deformation_contr}) in Lemma \ref{LEMMA_Willem} with $X=\scrH$, $I=J$, $c=\Fc$, $\eps=h/2$, $\rho=\nu/2$ (we decrease the value of $h$ if necessary so that $\mu \geq 8 h/\nu$). Therefore, there exists $\eta \in C([0,1]\times \scrH,\scrH)$ satisfying the properties of Lemma \ref{LEMMA_Willem}. Let $\overline{\gamma} \in \Gamma$ be such that
\begin{equation}\label{overline_gamma}
\max_{s \in [0,1]}J(\overline{\gamma}(s)) \leq \Fc+\frac{1}{4}h.
\end{equation}
Let us set $\hat{\gamma}: s \in [0,1] \to \eta(1,\overline{\gamma}(s)) \in \scrH$. Since $\eta(1,\cdot)$ is a homeomorphism by (iii) in Lemma \ref{LEMMA_Willem}, we have that $\hat{\gamma} \in C([0,1],\scrH)$, Moreover, by the definition of $h$ we have that $\Fc-h >\Fm$. Therefore, (i) in Lemma \ref{LEMMA_Willem} implies that for $i \in \{0,1\}$ we have $\hat{\gamma}(i)=\overline{\gamma}(i)\in \CV_i$. As a consequence, $\hat{\gamma} \in \Gamma$. Moreover, by (v) in Lemma \ref{LEMMA_Willem} and (\ref{overline_gamma}) we have that
\begin{equation}\label{hat_gamma}
\max_{s \in [0,1]}J(\hat{\gamma}(s)) \leq \Fc+\frac{1}{4}h
\end{equation}
which means by (ii) in Lemma \ref{LEMMA_Willem} that if $\hat{s} \in [0,1]$ is such that $J(\hat{\gamma}(\hat{s})) \geq \Fc$, then $\hat{\gamma}(s) \not \in F$, meaning that $(\psi+\hat{\gamma}(s))(\R) \cap K=\emptyset$. But since $\max_{s \in [0,1]}J(\hat{\gamma}(s)) <M$ by (\ref{hat_gamma}) and the definition of $h$, we get a contradiction with 2. in \ref{asu-K}, which we assume to hold true. Therefore, the proof is completed.
\qed
\end{proof}
Proposition \ref{PROPOSITION-local-PS} along with Lemmas \ref{LEMMA-PSbasic} and \ref{LEMMA-3rdwell} allows to finish the proof of Theorem \ref{THEOREM-general} as follows:

\textit{Proof of Theorem \ref{THEOREM-general} completed}. Assume that the hypothesis made for Theorem \ref{THEOREM-general} hold. Let $(u_n)_{n \in \N}$ and $(\tau_n)_{n \in \N}$ be the sequences given by Proposition \ref{PROPOSITION-local-PS}. By part 2 in Lemma \ref{LEMMA-PSbasic}, the sequence $(\tilde{u}_n):= (t_{\tau_n}(u_n))$ is a Palais-Smale sequence and it also satisfies
\begin{equation}\label{tilde_u-1}
\lim_{n \to \infty} \mathrm{dist}_\scrH(\tilde{u}_n(0)+\psi(0), K)=0.
\end{equation}
Up to an extraction, we have by \ref{asu-otherwells} that for all $n \in \N$ we have $J(\tilde{u}_n)\leq \tilde{C}:= (\Fm^\star-\Fc)/2+\Fc$. Therefore, by applying Lemma \ref{LEMMA-3rdwell}, we obtain $\rho_2:= \rho_2(\tilde{C})$ such that
\begin{equation}\label{tilde_u-2}
\forall n \in \N, \forall \sigma  \in \Sigma\setminus \{ \sigma^-,\sigma^+\},\forall t \in \R, \hspace{2mm} \lvert \tilde{u}_n(t)+\psi(t)-\sigma \rvert \geq \rho_2.
\end{equation}
Using now part 1 of Lemma \ref{LEMMA-PSbasic}, we find $\Fu \in H^1_{\loc}(\R,\R^k) \cap \CC^2(\R,\R^k)$ such that $\Fu$ solves (\ref{connections-eq}), $E(\Fu) \leq \Fc$ and for all $S_K \subset \R$ compact, $\tilde{u}_n+\psi \to \Fu$ in $H^1(S_K,\R^k)$ (in particular, $\tilde{u}_n \to \Fu$ pointwise in $\R$). Using (\ref{tilde_u-1}), the fact that $K$ is closed and pointwise convergence, we find $\Fu(0) \in K$. By assumption \ref{asu-K}, we have that $\Fu$ does not coincide with any minimizing heteroclinic in $\CF$. By (\ref{tilde_u-2}) and pointwise convergence, we have that
\begin{equation}\label{tilde_u-3}
\forall \sigma \in \Sigma \setminus \{ \sigma^-, \sigma^+\}, \forall t \in \R, \hspace{2mm} \lvert \Fu(t)-\sigma \rvert \geq \rho_2,
\end{equation}
meaning in particular that $\Fu$ cannot be a minimizing connecting orbit between $\sigma \in \Sigma \setminus \{ \sigma^-,\sigma^+\}$ and $\sigma' \in \{ \sigma^-,\sigma^+\}$. Assume now that $\Fu \in X(\sigma^-,\sigma^+)$. Due to the previous discussion, we must have $E(\Fu)>\Fm$. If $\Fu$ does not belong to $X(\sigma^-,\sigma^+)$, by Lemma \ref{LEMMA-limits-V} we have
\begin{equation}
\exists \sigma \in \{ \sigma^-,\sigma^+\}: \hspace{2mm} \lim_{t \to \pm \infty} \Fu(t) =\sigma
\end{equation}
and $\Fu(0) \not \in \{\sigma^-,\sigma^+\}$ because $\Fu(0) \in K$ and $K \cap \{\sigma^-,\sigma^+\}= \emptyset$ due to the first part of \ref{asu-K}. We also have that $\Fu(0) \not \in \Sigma \setminus \{\sigma^-,\sigma^+\}$ due to (\ref{tilde_u-3}). Therefore, $\Fu(0) \in \Sigma$. Hence, $\Fu$ is not constant.
\qed
\subsection{The proofs of Theorems \ref{THEOREM-symmetry} and \ref{THEOREM-symmetry-final}}\label{SUBS-SYM}
The first step of the proof of both Theorems consists on showing that there exists a Palais-Smale sequence $(u_n)_{n \in \N}$ at the level $\Fc_\sym$ such that $(\psi+u_n)_{n \in \N}$ approaches $X_{\sym,+}$. The existence of such sequence follows from the fact that we can map $X_\sym$ into $X_{\sym,+}$ continuously and leaving $X_{\sym,+}$ invariant and that such mapping does not increase the energy due to the symmetry assumption \ref{asu-symmetry}. The idea then is to show that a nontrivial solution is produced even if we have dichotomy of the Palais-Smale sequence. This proves Theorem \ref{THEOREM-symmetry}. More precisely, if a Palais-Smale sequence in $\scrH_{\sym,+}$ is not compact, then we are in the situation 1. of Theorem \ref{THEOREM-symmetry} and we find a pair of nontrivial homoclinic solutions. Of course, if such a Palais-Smale sequence is compact, we recover a solution in $X_{\sym,+}$ with energy $\Fc_\sym$, thus also nontrivial. Subsequently, for proving Theorem \ref{THEOREM-symmetry-final} under the additional assumption \ref{asu-finalK}, the argument is supplemented with a deformation argument analogous to that in the proof of Theorem \ref{THEOREM-general}. 

We begin by showing the following:
\begin{lemma}\label{LEMMA-FA}
Assume that \ref{asu-sigma}, \ref{asu-infinity}, \ref{asu-zeros} and \ref{asu-symmetry} hold. Let $d$ be as in (\ref{distance}) and $F_+:(X_\sym,d) \to (X_{\sym,+},d)$ be such that
\begin{equation}
\forall q \in X_\sym, \hspace{2mm} F_+(q)(t):= \begin{cases}
(\lvert q_1(t) \rvert, q_2(t),\ldots, q_k(t)) &\mbox{ if } t \geq 0,\\
(-\lvert q_1(t) \rvert, q_2(t),\ldots, q_k(t)) &\mbox{ if } t \leq 0.
\end{cases}
\end{equation}
Then for all $q \in X_\sym$ we have $E(F_+(q))\leq E(q)$, $F_+|_{X_{\sym,+}}=\mathrm{Id}|_{X_{\sym,+}}$ and $F_+$ is continuous.
\end{lemma}
\begin{proof}
Let $q \in X_{\sym}$, notice that repeating the argument in the proof of Lemma \ref{Lemma-Xsym} shows that $E(F_+(q))\leq E(q)$. Notice also that in case $q \in X_{\sym,+}$ then $F_+(q)=q$. Therefore, it only remains to show that $F_+$ is continuous. Let $(q_n)_{n \in \N}$ be a sequence in $X_\sym$ and $q \in X_\sym$ such that
\begin{equation}\label{FA-hyp}
\lim_{n \to \infty} \lVert q_n - q \rVert_\scrH = 0.
\end{equation}
For each $n \in \N$ set $q_n^+:= F_+(q_n) \in X_{\sym,+}$ and $q^+:= F_+(q) \in X_{\sym,+}$. We need to show that
\begin{equation}\label{FA-concl}
\lim_{n \to \infty} \lVert q_n^+-q^+\rVert_{\scrH}=0.
\end{equation}
Let $\kappa \leq \frac{1}{4}$ be arbitrary and take $t_q^+ \in \R$ such that
\begin{equation}\label{FA-id1}
\forall t \geq t_q^+, \hspace{2mm} \lvert q(t)-\sigma^+ \rvert \leq \kappa
\end{equation}
and $t_q^-<t_q^+$ such that
\begin{equation}\label{FA-id2}
\forall t \leq t_q^-, \hspace{2mm} \lvert q(t)-\sigma^- \rvert \leq \kappa.
\end{equation}
We set $I:= [t_v^-,t_v^+]$ By (\ref{FA-hyp}), we have that $q_n \to q$ uniformly, so in particular there exists $n_0 \in \N$ such that for all $n \geq n_0$ it holds $\lVert q_n-q \rVert_{L^{\infty}(\R,\R^k)} \leq \kappa$. This fact along with (\ref{FA-id1}), the definition of $\kappa$ and (\ref{FA-id2}) allow us to say that
\begin{equation}
\forall n \geq n_0, \forall t \in \R \setminus I, \hspace{2mm} q_n(t)=q_n^+(t) \mbox{ and } q(t)=q^+(t),
\end{equation}
which means that $(q_n^+ - q^+)_{n \in \N}$ converges to 0 in $H^1(\R\setminus I,\R^k)$ by (\ref{FA-hyp}). Hence, in order to establish (\ref{FA-concl}) we only need to show that $(q_n^+-q^+)_{n \in \N}$ converges to 0 in $H^1(I,\R^k)$. Notice that in fact all functions belong now to $H^1(I,\R^k)$ because $I$ is bounded. Let $f_+:H^1(I) \to H^1(I)$ the application such that
\begin{equation}
\forall v \in H^1(I), \forall t \in I, \hspace{2mm} f_+(v)(t):=\lvert v(t) \rvert.
\end{equation}
We have that the absolute value function is Lipschitz as a function from $\R$ to $\R$ and, moreover, the interval $I$ is bounded. Therefore, $f_+$ is continuous due to \textit{Theorem 1} in Marcus and Mizel \cite{marcus-mizel}. As a consequence, we have
\begin{equation}
\lim_{n \to \infty} \lVert q_{n,1}^+-q_{1}^+ \rVert_{H^1(I \cap [0,+\infty))}= \lim_{n \to \infty}\lVert f_+(q_{n,1})-f_+(q_1) \rVert_{H^1(I \cap [0,+\infty))}=0
\end{equation}
and
\begin{equation}
\lim_{n \to \infty} \lVert q_{n,1}^+-q_{1}^+ \rVert_{H^1(I \cap (-\infty,0])}= \lim_{n \to \infty}\lVert -f_+(q_{n,1})+f_+(q_1) \rVert_{H^1(I \cap (-\infty,0])}=0,
\end{equation}
that is
\begin{equation}
\lim_{n \to \infty} \lVert q_{n,1}^+-q_{1}^+ \rVert_{H^1(I)} = 0.
\end{equation}
Since all the other components were not modified, (\ref{FA-concl}) has been proven and the proof is concluded.
\qed
\end{proof}

Lemma \ref{LEMMA-FA} implies the following:
\begin{lemma}\label{LEMMA-hsym}
Assume that \ref{asu-sigma}, \ref{asu-infinity}, \ref{asu-zeros}, \ref{asu-heteroclinics} and \ref{asu-symmetry} hold. Let $h_\sym: \scrH_\sym \to \scrH_\sym$ be defined as
\begin{equation}
h_\sym: v \in \scrH_\sym \to F_+(v+\psi)-\psi \in \scrH_\sym.
\end{equation} 
Then for all $v \in \scrH_\sym$ we have $h_\sym(v)+\psi \in X_{\sym,+}$, $J(h_\sym(v))\leq J(v)$ and for all $\gamma \in \Gamma_{\sym}$ it holds that the composed path $h_\sym \circ \gamma$ belongs to $\Gamma_{\sym}$.
\end{lemma}
\begin{proof}
Let $v \in \scrH$. By Lemma \ref{LEMMA-FA} we have that $h_\sym(v)+\psi=F_+(v+\psi) \in X_{\sym,+}$ and $J(h_\sym(v))=E(F_+(v+\psi)) \leq E(v+\psi) = J(v)$. It is straightforward to show that $h_\sym$ is continuous. Notice that if $\Fv \in \CV_\sym$ then $\CV_\sym+\{\psi\} = \CF_\sym \subset X_{\sym,+}$ by definition. Therefore, using again Lemma \ref{LEMMA-FA} we have $h_\sym(\Fv)=F_+(\Fv+\psi)-\psi=\Fv+\psi-\psi=\Fv = \mathrm{Id}_{\scrH_\sym}(\Fv)$.
\qed
\end{proof}

\subsubsection{The proof of Theorem \ref{THEOREM-symmetry}}
We have the following result:
\begin{proposition}\label{PROPOSITION-sym-+}
Assume that \ref{asu-sigma}, \ref{asu-infinity}, \ref{asu-zeros}, \ref{asu-heteroclinics} and \ref{asu-symmetry} hold. Then, there exists a sequence $(u_n)_{n \in \N}$ in $\scrH_{\sym}$ such that 
\begin{equation}\label{PS-sym}
\lim_{n \to +\infty} J(u_n) \to \Fc_\sym \mbox{ and } DJ(u_n) \to 0 \mbox{ in } \scrH_\sym.
\end{equation}
and, moreover,
\begin{equation}\label{PS-sym2}
\lim_{n \to +\infty}d(u_n+\psi,X_{\sym,+}) =0,
\end{equation}
where $d$ is as in (\ref{distance}).
\end{proposition}
The proof of Proposition \ref{PROPOSITION-sym-+} is a direct consequence of Proposition \ref{PROPOSITION-sym-MP} and  Lemma \ref{LEMMA-hsym} along with a usual variant of the mountain pass lemma (see for instance \textit{Corollary 4.3} in Mawhin and Willem \cite{mawhin-willem}) which allows to find a Palais-Smale sequence associated to any given min-maxing sequence of paths. We can now tackle the final part of the proof of Theorem \ref{THEOREM-symmetry}.

\textit{Proof of Theorem \ref{THEOREM-symmetry} completed}. Assume that the hypothesis of Theorem \ref{THEOREM-symmetry} hold. Let $(u_n)_{n \in \N}$ be the Palais-Smale sequence provided by Proposition \ref{PROPOSITION-sym-+}. By assumption \ref{asu-otherwells-sym}, up to an extraction we have
\begin{equation}
\sup_{n \in \N}J(u_n) \leq C < \Fm^\star
\end{equation}
 for an arbitrary $C \in (\Fc_\sym,\Fm^\star)$. We can then use Lemma \ref{LEMMA-3rdwell} to find $\rho_2>0$ such that
\begin{equation}\label{SYM1-proof-3rdwell}
\forall n \in \N, \forall \sigma \in \Sigma \setminus \{ \sigma^-,\sigma^+\}, \forall t \in \R, \hspace{2mm} \lvert u_n(t)+\psi(t)-\sigma \rvert \geq \rho_2.
\end{equation}
We divide the proof according to the two possible scenarios (dichotomy or compactness):

\textbf{Case 1}. Dichotomy. Assume that there exist $c_0>0$, $c_1>0$  and a sequence $t_n \to \infty$ such that, up to an extraction
\begin{equation}\label{SYM1-Case1}
\forall n \in \N, \hspace{2mm} \int_{t_n-c_1}^{t_n+c_1} e(u_n+\psi) \geq c_0.
\end{equation}
Since $(u_n+\psi)_{n \in \N}$ approaches $X_{\sym,+}$ due to (\ref{PS-sym2}), up to an extraction we can suppose
\begin{equation}\label{SYM1-1-1}
\forall n \in \N, \forall t \geq 0, \hspace{2mm} \lvert u_n(t)+\psi(t)-\sigma^- \rvert \geq \rho_2.
\end{equation}
For each $n \in \N$, we can define $\tilde{q}_n:= u_n(\cdot+t_n)+\psi(\cdot+t_n) \in X(\sigma^-,\sigma^+)$ and $\tilde{u}_n:= \tilde{q}_n-\psi$. We can regard $(u_n)_{n \in \N}$ as a Palais-Smale sequence in $\scrH$ because $\scrH_\sym$ is a closed subspace of $\scrH$. Part 2 in Lemma \ref{LEMMA-PSbasic} implies then that $(\tilde{u}_n)_{n \in \N}$ is a Palais-Smale sequence in $\scrH$. By using now part 1 of Lemma \ref{LEMMA-PSbasic}, we find $\Fu_+ \in H^1_{\loc}(\R,\R^k)$ such that for all $S_K \subset \R$ compact, $\tilde{q}_n \to \Fu_+$ in $H^1(S_K,\R^k)$. Moreover, $\Fu_+ \in \CC^2(\R,\R^k)$ solves (\ref{connections-eq}) and $E(\Fu_+) \leq \Fc_\sym$. By (\ref{SYM1-Case1}) and the convergence, we have
\begin{equation}
\int_{-c_1}^{c_1}  e(\Fu_+) = \lim_{n \to +\infty} \int_{-c_1}^{c_1} e(\tilde{q}_n) \geq c_0,
\end{equation}
meaning that $E(\Fu_+) \geq c_0$, so in particular $\Fu_+$ is not constant. We now show that $\Fu_+$ converges to $\sigma^+$ at infinity. Rewriting (\ref{SYM1-1-1}) for $(\tilde{q}_n)_{n \in \N}$, we have
\begin{equation}\label{SYM-1-2}
\forall n \in \N, \forall t \geq -t_n, \hspace{2mm} \lvert \tilde{q}_n(t)-\sigma^- \rvert \geq \rho_2,
\end{equation}
which combined with (\ref{SYM1-proof-3rdwell}), Lemma \ref{LEMMA-limits-V} and pointwise convergence $\tilde{q}_n \to \Fu_+$ gives that
$\lim_{t\to \pm \infty} \Fu_+(t)=\sigma^+$ as we wanted. Finally, notice that by symmetry we have that the function
\begin{equation}
\Fu_-: t\in \R \to \Fs(\Fu_+(t)),
\end{equation}
is a non constant solution of (\ref{connections-eq}) such that $\lim_{t \to \pm \infty} \Fu_-(t)=\sigma^-$. 

\textbf{Case 2}. Compactness. The hypothesis made for Case 1 is not satisfied. Then, for all $c_2>0$ there exists $t(c_2)>0$ such that
\begin{equation}
\forall n \in \N, \hspace{2mm} \int_{t(c_2)}^{+\infty} e(u_n+\psi) \leq c_2
\end{equation}
and, by symmetry
\begin{equation}
\forall n \in \N, \hspace{2mm} \int_{-\infty}^{-t(c_2)}e(u_n+\psi) \leq c_2.
\end{equation}
Equivalently, up to taking a diagonal extraction, for each $m \geq 1$ we can find $t(m) \geq 0$ such that
\begin{equation}\label{SYM1-2-1}
\forall n \in \N, \hspace{2mm} \int_{-t(m)}^{t(m)} e(u_n+\psi) \geq \Fc_\sym-\frac{1}{m}.
\end{equation}
Using again Lemma \ref{LEMMA-PSbasic}, we find $\Fu \in H^1_{\loc}(\R,\R^k)$ a solution to (\ref{connections-eq}) such that $u_n+\psi \to \Fu$ strongly in $H^1(S_K,\R^k)$ for each compact interval $I$. Moreover, by (\ref{SYM1-proof-3rdwell}) and (\ref{PS-sym2}) we have $\Fu \in X_{\sym,+}$. Finally, using (\ref{SYM1-2-1}) we get $E(\Fu)=\Fc_\sym$, which concludes the proof.
\qed
\subsubsection{The proof of Theorem \ref{THEOREM-symmetry-final}}
We will use \ref{asu-finalK} and Lemma \ref{LEMMA_Willem}. Define
\begin{equation}
A_\sym:= \{ q \in X_\sym: q(0) \in K_\sym \mbox{ and } E(q) \geq \Fc_\sym\}
\end{equation}
and
\begin{equation}\label{Fsym}
F_\sym:= A_\sym-\{\psi\} \in \scrH_\sym.
\end{equation}
We have the following, which is the analogous of Proposition \ref{PROPOSITION-local-PS}:
\begin{proposition}\label{PROPOSITION-PS2}
Assume that \ref{asu-sigma}, \ref{asu-infinity}, \ref{asu-zeros}, \ref{asu-heteroclinics}, \ref{asu-symmetry} and \ref{asu-finalK} hold. Then, there exists a sequence $(u_n)_{n \in \N}$ in $\scrH_{\sym}$ such that 
\begin{equation}\label{PS2-sym}
\lim_{n \to +\infty} J(u_n) \to \Fc_\sym \mbox{ and } DJ(u_n) \to 0 \mbox{ in } \scrH_\sym
\end{equation}
and, moreover,
\begin{equation}\label{PS2-sym2}
\lim_{n \to +\infty}\mathrm{dist}_\scrH(u_n+\psi,X_{\sym,+}\cap A_\sym) =0.
\end{equation}
\end{proposition}
The proof of Proposition \ref{PROPOSITION-PS2} is analogous to the proof of Proposition \ref{PROPOSITION-local-PS}. The only significant difference is that the path which is obtained from the deformation provided by Lemma \ref{LEMMA_Willem} must be contained in $X_{\sym.,+}$ in order to get the contradiction with \ref{asu-finalK}. However, this can be assumed by Lemma \ref{LEMMA-hsym}. Hence, we do not include the proof of Proposition \ref{PROPOSITION-PS2} here.

\textit{Proof of Theorem \ref{THEOREM-symmetry-final} completed}. We now suppose that the assumptions of Theorem \ref{THEOREM-symmetry-final} are satisfied.  Let $(u_n)_{n \in \N}$ be the Palais-Smale sequence given by Proposition \ref{PROPOSITION-PS2}. As done before, up to an extraction we can use \ref{asu-otherwells-sym} and Lemma \ref{LEMMA-3rdwell} to find $\rho_2>0$ such that
\begin{equation}\label{SYM2-proof-3rdwell}
\forall n \in \N, \forall \sigma \in \Sigma \setminus \{ \sigma^-,\sigma^+\}, \forall t \in \R, \hspace{2mm} \lvert u_n(t)+\psi(t)-\sigma \rvert \geq \rho_2.
\end{equation}
Regarding $(u_n)_{n \in \N}$ as a Palais-Smale sequence in $\scrH_\sym$ and using Lemma \ref{LEMMA-PSbasic}, we find $\Fu \in H^1_{\loc}(\R,\R^k)$ such that $E(\Fu) \leq \Fc_\sym$, $u_n+\psi \to \Fu$ strongly in $H^1(S_K,\R^k)$ ($S_K$ compact). By (\ref{SYM2-proof-3rdwell}), we have that
\begin{equation}\label{SYM2-proof-1}
\forall \sigma \in \Sigma \setminus \{ \sigma^-,\sigma^+\}, \forall t \in \R, \hspace{2mm} \lvert \Fu(t) - \sigma \rvert \geq \rho_2.
\end{equation}
By pointwise convergence, we have for all $t \in \R$, that $\Fu(-t)=\Fs(\Fu(t))$. Since $u_n+\psi$ approaches $X_{\sym,+}$ due to (\ref{PS2-sym2}), we have for all $t \geq 0$, $\Fu_1(t) \geq 0$ and analogously for $t \leq 0$. These facts along with (\ref{SYM2-proof-1}) give $\lim_{t \to \pm \infty}\Fu(t)=\sigma^\pm$, which all together implies $\Fu \in X_{\sym,+}$. Finally, using again (\ref{PS-sym2}) we have $\Fu(0) \in K_\sym$, which by \ref{asu-finalK} means that $\Fu \not \in \CF_\sym$, i.e., $E(\Fu) > \Fm$.
\qed
\section{On the assumptions \ref{asu-K} and \ref{asu-finalK}}\label{section-asu}
As commented in subsection \ref{SUBS-previous}, assumptions \ref{asu-K} and \ref{asu-finalK} might appear as rather artificial and, moreover, difficult to verify in hypothetical applications. Despite the fact that in Theorem \ref{THEOREM-symmetry} we show that \ref{asu-K} can be removed if we restrict to potentials which are symmetric as in \ref{asu-symmetry}, we believe that a better understanding of \ref{asu-K} is still an interesting open question. Indeed, even though adding symmetry is a natural procedure in order to simplify a problem, it can be found to be too restrictive in many applications. In this direction, we show in Lemma \ref{LEMMA-3m} that \ref{asu-K} holds if the mountain pass value $\Fc$ lies outside some known countable subset of $(\Fm,+\infty)$, and in particular if it is smaller than $3\Fm$. As explained in subsection \ref{SUBS-previous}, this requirement is equivalent to the assumption made by Bisgard in \cite{bisgard}. In any case, a better understanding of hypothesis \ref{asu-K} and \ref{asu-finalK} remains an open problem. Geometric intuition suggests that such hypothesis should always (or close) hold, but we do not have a proof of such a fact. The same type of comment is made by Bisgard in \cite{bisgard}, where he states (see the Remark after his \textit{Theorem 2.3})  that he expects his assumption on $\Fc$ to be generic (that is, valid for a dense class of potentials). We also think that this is the natural conjecture as the set of \textit{bad} values for $\Fc$ is discrete.We believe that a starting point to aim at understanding this question better would be to try to understand the relation between the mountain pass value and the geometry of $V$ in a deeper fashion.

We now state the result which links \ref{asu-K} and Bisgard's assumption:
\begin{lemma}\label{LEMMA-3m}
Let $V$ be a potential satisfying \ref{asu-sigma}, \ref{asu-infinity}, \ref{asu-zeros} and \ref{asu-heteroclinics}. Let $\Fc$ be the mountain pass value defined in (\ref{Fc}). Then, if we have
\begin{equation}\label{2k+1m}
\Fc \in (\Fm,+\infty)\setminus \{ (2j+1)\Fm: j \in \N^*\},
\end{equation}
there exists $K \subset \R^k$ such that assumption \ref{asu-K} is satisfied for some constants $\nu_0>0$ and $M > \Fc$.
\end{lemma}
\begin{proof}
For each $\eps>0$, define
\begin{equation}\label{K_eps}
\tilde{K}_\eps:= \bigcup_{\Fq \in \CF} \{ \bu \in \R^k: \mathrm{dist}(\bu,\Fq(\R)) < \eps\}
\end{equation}
and $K_\eps:= \R^k \setminus \tilde{K}_\eps$. The proof will be concluded if we show the existence of $\nu_0>0$ and $M>0$ such that for any $\gamma \in \Gamma$, with $\max_{s \in [0,1]}J(\gamma(s)) \leq M$ there exists $s_\gamma \in [0,1]$ such that $(\gamma(s_\gamma)+\psi)(\R) \cap K_{\nu_0} \not = \emptyset$ and $J(\gamma(s_\gamma)) \geq \Fc$. By contradiction, assume that for any $\eps >0$ and $M > \Fc$, there exists $\gamma_\eps \in \Gamma$ with $\max_{s \in [0,1]}J(\gamma_\eps(s)) \leq M$ such that for all $s \in [0,1]$ satisfying $(\gamma_\eps(s)+\psi)(\R) \cap K_\eps \not = \emptyset$ we have $J(\gamma_\eps(s)) < \Fc$. Otherwise stated, if $s \in [0,1]$ is such that $J(\gamma_\eps(s)) \geq \Fc$, then $(\gamma_\eps(s)+\psi)(\R) \subset \tilde{K}_\eps$. Taking subsequences $(\eps_n)_{n \in \N}$ and $(M_n)_{n \in \N}$ such that $\eps_n \to 0^+$ and $M_n \to \Fc^+$ as $n \to \infty$, we have found a sequence of paths $(\gamma_{\eps_n})_{n \in \N}$ such that $\max_{s \in [0,1]}J(\gamma_{\eps_n}(s)) \to \Fc$. By usual arguments (for instance \textit{Corollary 4.3} in Mawhin and Willem \cite{mawhin-willem}), we find a Palais-Smale sequence $(v_{\eps_n})_{n \in \N}$ at the level $\Fc$ such that
\begin{equation}\label{lim-epsn-1}
\lim_{n \to \infty} \mathrm{dist}_{\scrH}(v_{\eps_n},\{ v \in \scrH: J(v) \geq \Fc\} \cap \gamma_{\eps_n}([0,1])) \to 0.
\end{equation}
Due to the contradiction assumption stated above, we have that if
\begin{equation}
v \in \{ v \in \scrH: J(v) \geq \Fc\} \cap \gamma_{\eps_n}([0,1]),
\end{equation}
then $(v+\psi)(\R) \subset \tilde{K}_{\eps_n}$. The goal now is to obtain that $\Fc = (2j+1)\Fm$ for some $j \in \N^*$, which will give the desired contradiction since we assume (\ref{2k+1m}). Let $s_\tau:= (\tau_n)_{n \in \N}$ be an arbitrary sequence in $\R$. Using Lemma \ref{LEMMA-PSbasic}, we have that $(v_n^{\tau_n})_{n \in \N}$  (with the notations as in the second part of Lemma \ref{LEMMA-PSbasic}) is a Palais-Smale sequence at the level $\Fc$ converging (up to subsequences) locally in $H^1$ to $\Fq^{s_\tau}$ a solution of (\ref{connections-eq}) with $E(\Fq^{s_\tau}) \leq \Fc$. Using (\ref{lim-epsn-1}), we have that in fact $\Fq^{s_\tau}$ is either a constant equal to $\sigma^-$ or $\sigma^+$, $\Fq^{s_\tau} \in \CF$ or $\Fq^{s_\tau}(-\cdot) \in \CF$. Therefore, by Proposition \ref{PROPOSITION_ASYMPTOTIC_PS} it follows that there exists $j \in \N^*$ and sequences $((t_n^0,\ldots,t_n^{2j+1}))_{n \in \N}$, $(\tau_n^1,\ldots,\tau_n^{2j+2}))_{n \in \N}$  in $\R^{2j+2}$ and $\R^{2j+1}$ respectively  such that (up to an extraction)
\begin{equation}
\forall j' \in \{1,\ldots,2j+2\}, \hspace{2mm}  t_n^{j'-1} - t_n^{j'} \to +\infty \mbox{ as } n \to \infty,
\end{equation}
\begin{align}
\forall j' \in \{0,\ldots,j\},\hspace{2mm}& \lim_{n \to \infty}v_n(t^{2j'}_n)+\psi(t^{2j'}_n)=\sigma^-,\\ &\lim_{n \to \infty}v_n(t^{2j'+1}_n)+\psi(t^{2j'+1}_n)=\sigma^+,
\end{align}
\begin{align}
\forall j' \in \{ 1, \ldots, 2j+2\}, \hspace{2mm}&(v_n+\psi) - \Fq^{j'}(\cdot-\tau_n^{j'}) \to 0\\ &\mbox{ strongly in } H^1([t_n^{j'},t_n^{j'+1}],\R^k) \mbox{ as } n \to \infty
\end{align}
with $\Fq^{j'}  \in \CF$  if $j'$ is even and $\Fq^{j'} (-\cdot) \in \CF$ if $j'$ is odd. Moreover
\begin{equation}
\lim_{n \to \infty}E(v_n+\psi)=\lim_{n \to +\infty}\sum_{j'=1}^{2j+2}\int_{t^{j'-1}_n}^{t_n^{j'}}e(v_n+\psi)=(2j+1)\Fm
\end{equation}
which gives the desired contradiction.
\qed
\end{proof}
Notice that if $\Fc<3\Fm$, then (\ref{2k+1m}) holds.
 \begin{remark}\label{REMARK-3m}
An interpretation of Lemma \ref{LEMMA-3m} can be given as follows: Take a function $q$ which has energy strictly greater than $\Fm$, $DJ$ applied to $q-\psi$ has \textit{small} norm and the trace of $q$ is \textit{close} enough to the traces of the elements of $\CF$.  Then, $q$ must look close to one element of $\CF$ which is glued to $j \geq 1$ \textit{cycles} in $\CF$. Such cycles are as follows: take an element of $\CF$ and glue it to an element of $\CF$ with reversed sign to obtain a connecting orbit joining $\sigma^-$ and $\sigma^+$. The energy of $q$ must be then close to $(2j+1)\Fm$. This argument is the key of the proof of Lemma \ref{LEMMA-3m}. An illustration is shown in Figure \ref{FIGURE-3m}. In different words words, Palais-Smale sequences which have the type of behavior described above yield only trivial solutions. The point of assumptions \ref{asu-K} and \ref{asu-finalK} is to exclude such type of behaviors for Palais-Smale sequences.
\end{remark}
\begin{figure}\label{FIGURE-3m}
\centering
\includegraphics[scale=0.6]{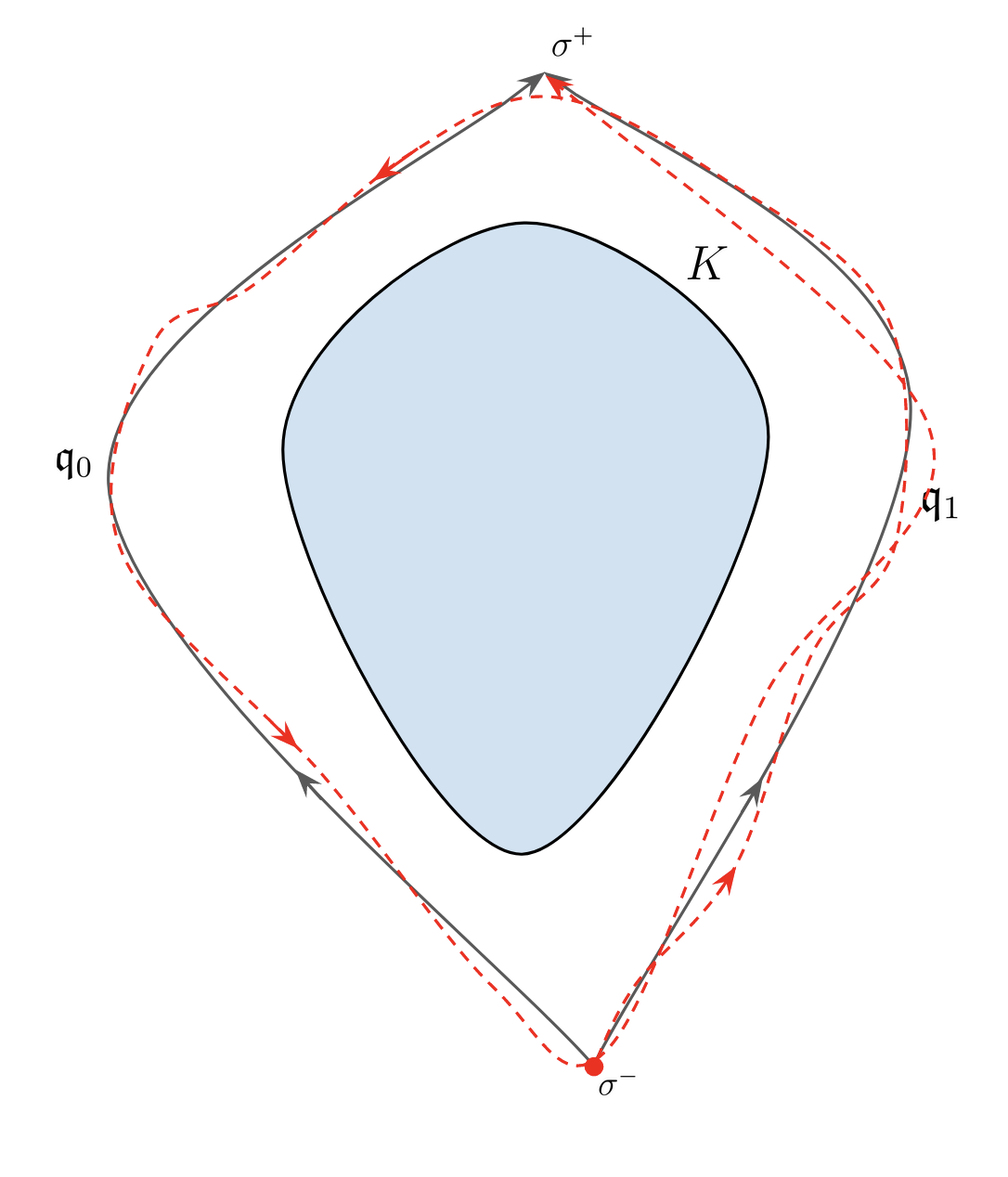}
\caption{Illustration of Remark \ref{REMARK-3m}. The functions $\Fq_0$ and $\Fq_1$ represent two globally minimizing heteroclinics joining $\sigma^-$ and $\sigma^+$. The set $K$ is away from the traces of $\Fq_0$ and $\Fq_1$. The discontinuous curve represents the function with the \textit{bad} behavior that we want to avoid by introducing \ref{asu-K} and \ref{asu-finalK}. In particular, this behavior is excluded if the mountain pass level is below the minimum energy necessary for a function to behave like the discontinuous curve (see Lemma \ref{LEMMA-3m}).}
\end{figure}  
We obtain the analogous result for symmetric potentials, with an identical proof:
\begin{lemma}\label{LEMMA-3m-sym}
Let $V$ be a potential satisfying \ref{asu-sigma}, \ref{asu-infinity}, \ref{asu-zeros}, \ref{asu-heteroclinics} and \ref{asu-symmetry}. Let $\Fc_\sym$ be the mountain pass value defined in (\ref{Fc-sym}). Then, if we have
\begin{equation}\label{2k+1m_sym}
\Fc_\sym \in (\Fm,+\infty)\setminus \{ (2j+1)\Fm: j \in \N^*\},
\end{equation}
there exists $K_\sym \subset \R^k$ such that assumption \ref{asu-finalK} is satisfied for some constants $\nu_0>0$ and $M > \Fc$.
\end{lemma}

\begin{acknowledgements}
I wish to thank my PhD advisor Fabrice Bethuel for bringing this problem into my attention and for many useful comments and remarks during the elaboration of this paper. I also wish to thank the referee for pointing to several important references such as \cite{bisgard} as well as for numerous remarks and suggestions which lead to significant improvements on the paper.

\includegraphics[scale=0.3]{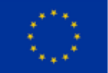}
This program has received funding from the European Union's Horizon 2020 research and innovation programme under the Marie Skłodowska-Curie grant agreement No 754362.
\end{acknowledgements}

\end{document}